\documentclass[letter]{article}
\usepackage[utf8x]{inputenc}
\usepackage{amsfonts,vmargin,graphicx,amssymb}
\usepackage{amsmath,amsthm,color}

\newtheorem{theorem}{Theorem}[section]

\newtheorem{lemma}{Lemma}[section]
\newtheorem{corollary}{Corollary}[section]
\newtheorem{remark}{Remark}[section]

\newcommand{\e}[1]{\begin{equation}#1\end{equation}}

\newcommand{\dx}{\partial_x}
\newcommand{\dy}{\partial_y}
\newcommand{\lu}[2]{\lambda_{#1,#2}}
\newcommand{\luoz}{\lambda_{1,0}}
\newcommand{\luzo}{\lambda_{0,1}}
\newcommand{\luij}{\lambda_{i,j}}
\newcommand{\dux}{\partial_x}
\newcommand{\duy}{\partial_y}

\newcommand{\eye}{{\rm{}i}}
\newcommand{\bfn}{\mathbf{n}}
\newcommand{\jmp}[1]{[\![#1]\!]}  
\newcommand{\mv}[1]{\{\!\!\{#1\}\!\!\}}           
\newcommand{\avg}[1]{\{\!\!\{#1\}\!\!\}}           
\newcommand{\cEI}{{\cal E}_I}
\newcommand{\cED}{{\cal E}_D}
\newcommand{\cER}{{\cal E}_R}
\newcommand{\cE}{{\cal E}}

\newcommand{\GPWK}{GPW_\kappa^{p,q}(K)}
\newcommand{\GPWKhat}{GPW_\kappa^{p,q}\left(\hat K\right)}

\title{Numerical simulation of wave propagation in inhomogeneous media using Generalized Plane Waves}
\author{Lise-Marie Imbert-G\'erard\thanks{Courant Institute of Mathematical Sciences,
New York University,
251 Mercer Street,
New York, N.Y. 10012-1185, USA. (e-mail: imbertgerard@cims.nyu.edu)}, Peter Monk\thanks{Department of Mathematical Sciences, University of Delaware, Newark DE 19716, USA  (e-mail: monk@udel.edu).}
}

\begin{document}
\maketitle
\tableofcontents

\section*{Abstract}
The Trefftz Discontinuous Galerkin (TDG) method is a technique for approximating the Helmholtz equation (or other linear wave equations) using
piecewise defined local solutions of the equation to approximate the global solution.  When coefficients in the equation (for example, the refractive index) are piecewise constant it is common to use plane waves on each element.  However when the coefficients are smooth functions of position, plane waves are no longer directly applicable.  In this paper we show how Generalized Plane Waves (GPWs) can be used in a modified TDG scheme to approximate the solution for piecewise smooth coefficients.  GPWs are approximate solutions to the equation that reduce to plane waves when the medium through which the wave propagates is constant.  We shall show how to modify the TDG sesquilinear form to allow us to prove 
convergence of the GPW based version.  The new scheme retains the high order convergence of the original TDG scheme (when the solution
is smooth) and also retains the same number of degrees of freedom per element (corresponding to the directions of the GPWs).  Unfortunately it looses the advantage that only skeleton integrals need to be performed.  Besides proving convergence, we provide numerical examples to test our theory.

\section{Introduction}
The Trefftz Discontinuous Galerkin (TDG) method proposed in \cite{git09} is a mesh based method for approximating solutions of the Helmholtz
equation.  This method generalizes the Ultra Weak Variational Formulation (UWVF)~of the same problem~\cite{cessenat_phd,despres} by allowing different weighting
strategies on penalty terms in the TDG method.  Error analysis \cite{buf07,git09,HMP11,hmp13,hmp15}  and computational experience \cite{hut03} 
show that the method can be an efficient way of approximating solutions of the Helmholtz equation.  It has also become clear that the method
works best in an $hp$-mode (see  \cite{hmp15}) where large elements are used away from boundaries in the computational domain, together with larger numbers of plane waves.

However because of the use of simple Trefftz functions (usually plane waves element by element), it has to be assumed that the coefficients in the governing partial differential
equation are piecewise constant.  Of course smoothly varying coefficient functions could be first approximated by a piecewise constant function and then the resulting perturbed Helmholtz equation could be solved by TDG or UWVF.  But this would require small elements and hence defeat some of the potential advantages of using large elements in a Trefftz based scheme.

To circumvent the difficulty with  smoothly varying coefficients, we propose to use approximate solutions of the underlying partial differential equation constructed
element by element where the coefficients are variable.  In this work we use the Generalized Plane Waves (GPW) of \cite{IGD2013,IG2015} as a basis for the TDG type scheme.   

Note that smoothly varying coefficients arise in the simulation of electromagnetic wave propagation in tokamaks where the permittivity is spatially variable and may even become negative.  Indeed the original design of GPWs in \cite{IGD2013,IG2015} was motivated precisely by this application.

To describe the setting for applying GPWs in more detail
let us consider the following model problem from \cite{hmp13}.  Given a bounded Lipschitz polyhedron $\Omega$ that is star shaped 
with respect to the origin and a larger Lipschitz polyhedron $\Omega_R$ containing $\overline{\Omega_D}$, we define the computational domain to be the annulus 
$\Omega=\Omega_R\setminus\overline{\Omega_D}$.  The two boundaries of $\Omega$ are $\Gamma_D=\partial\Omega_D$ and $\Gamma_R=\partial\Omega_R$ and we use a normal $\mathbf{n}$ that is outwards from $\Omega$. Because we shall use some regularity results from
\cite{hmp13} we need to assume that $\Omega_R$ is star-like with respect to a ball of radius $\gamma_Rd_\Omega$ centered at the origin where
$\gamma_R>0$ and $d_\Omega$ is the diameter of $\Omega$.

Suppose we are give a wave number $\kappa>0$.  In addition given a strictly positive, piecewise smooth and bounded real function $\epsilon\in L^{\infty}(\Omega)$  and another function $g_R\in L^2(\Gamma_R)$, we want to approximate
the solution $u$ of
\begin{eqnarray}
\Delta u+\kappa^2  \epsilon u&=&0\mbox{ in }\Omega,\label{helm}\\
u&=&0\mbox{ on }\Gamma_D,\label{dirich}\\
{\partial_n u}+\eye\kappa u&=&\eye\kappa g\mbox{ on }\Gamma_R.\label{imp}
\end{eqnarray}
As pointed out in \cite{hmp13} this is a model problem for scattering (for example of an $s$-polarized electromagnetic wave from a perfect conductor embedded in a dielectric  in two dimensions).  The impedance boundary condition is then a simple radiation boundary condition.

As we shall detail shortly, if we assume that the function $\epsilon$ is piecewise analytic on each element, we can approximate it by a power series.  With this in hand we shall give details of a recursive algorithm for generating the coefficients of basis functions on each element that satisfy (\ref{helm}) to high accuracy.  These are constructed so that if the coefficient $\epsilon$ is constant on an element, the resulting basis function is just a plane wave.  Using these generalized plane waves we can prove convergence of a modified TDG scheme.  The resulting discrete problem obtains high order convergence
for smooth solutions as is the case for the standard TDG or UWVF.  We only consider $h$-convergence in this paper.

As we shall see, the main disadvantage of the use of GPWs in the TDG method is the need to integrate over elements in the grid.  We would prefer
to use them in a generalized UWVF avoiding this integration.  However, although numerical experiments are encouraging \cite{LM_thesis} we do not have a theoretical justification of this approach.  

The paper proceeds as follows.  In the next section we briefly outline our modification of the basic TDG method.  Then in Section~\ref{GPW} we show how to construct GPWs and then obtain two new error estimates  for these functions that will underlie our error analysis.  We also show that piecewise linear functions can be approximated element by element using the GPW functions.  In Section~\ref{EE} we derive
error estimates for the new TDG scheme with GPW basis functions.  Finally in Section~\ref{NT} we give some basic numerical tests of the new algorithm.

\section{The Plane Wave Discontinuous Galerkin Method}
Even with a variable coefficient $\epsilon$, the choice of domain $\Omega$ and the conditions on $\epsilon$ guarantee that if $g_R\in H^s(\Gamma_R)$ for some sufficiently small $s>0$ (depending on the interior angles of $\Gamma_R$) then $u\in H^{3/2+s}(\Omega)$~\cite[Theorem 2.3]{hmp13}. This regularity will allow us to 
develop consistent fluxes for $u$ and it's derivatives.  Unfortunately, because $\epsilon$ is variable, the dependence on $\kappa$ of the continuity
constants for estimates in this paper is not easy to track.  Therefore we note now that constants in the analysis will depend in an unspecified way on $\kappa$.

As is usual for DG schemes, we start with a mesh and continue to define the method using definitions from \cite{hmp13}. Suppose we cover $\Omega$ by a finite element mesh ${\cal T}_h$ of regular triangular elements $K$ of maximum diameter $h$ (in fact more general domains can easily be allowed). The diameter of an element $K$ is denoted $h_K$. In addition following \cite{hmp13} we assume
\begin{enumerate}
\item {\em Local quasi-uniformity}: there exists a constant $\tau\geq 1$ independent of $h$ such that
\[
\tau^{-1}\leq \frac{h_{K_1}}{h_{K_2}}\leq \tau\mbox{ for all triangles }K_1,K_2\mbox{ meeting at any edge}.
\]
\item {\em Quasi-uniformity close to $\Gamma_R$}: For all triangles $K$ touching $\Gamma_R$ there is a constant $\tau_R$ independent of $h$ such that
\[
\frac{h}{h_K}\leq \tau_R.
\]
\end{enumerate}
Let $\bfn_K$ denote the unit outward normal to element $K$.
Let 
 $K$ and $K'$ denote two elements in ${\cal T}_h$ meeting at  an  edge $e$ then, on $e$, we make the standard definitions of the average value and jump of functions across $f$:
\begin{eqnarray*}
\mv{u}=\frac{u|_{K}+u|_{K'}}{2},&\quad&\mv{\mathbf{\sigma}}=\frac{\mathbf{\sigma}|_K+\mathbf{\sigma}|_{K'}}{2},\\
\jmp{u}=u|_K\bfn_K+u|_{K'}\bfn_{K'},&\quad&
\jmp{\mathbf{\sigma}}=\mathbf{\sigma}|_{K}\cdot\bfn_K+\mathbf{\sigma}|_{K'}\cdot\bfn_{K'}.
\end{eqnarray*}
We denote by $\cE_h$ the set of all edges in the mesh.  Then let
\begin{itemize}
\item $\cEI$ denote the set of all edges in the mesh interior to $\Omega$,
\item $\cED$ is the set of all boundary edges on $\Gamma_D$,
\item  $\cER$ is the set of all boundary edges on $\Gamma_R$.
\end{itemize}
We also need three positive penalty parameters that are functions of position on 
the skeleton of the mesh: $\alpha$, $\beta$ and $\delta$.  At this point these are simply assumed to be positive functions of position on $\cE$ and will be given in more detail shortly.
Using the above defined jumps and average values, we are lead to consider the following standard sesquilinear form for TDG \cite{hmp13,MelenkEsterhazy12,kapita14}:
\begin{eqnarray}
A_h(u,v)&=& \int_{\Omega}\left(\nabla_hu\cdot\nabla_h\overline{v}
-\kappa^2\epsilon u\,\overline{v}\right)\,dA-\int_{\cEI}\left(\avg{\nabla_h u}\cdot\jmp{\overline{v}}
+\jmp{ u}\cdot\avg{\nabla_h\overline{v}}\right)\,ds\nonumber\\&&
-\frac{1}{i\kappa}\int_{\cEI}\beta\jmp{\nabla_h u}\jmp{\nabla_h\overline{v}}\,ds
+{i\kappa}\int_{\cEI}\alpha\jmp{ u}\cdot\jmp{\overline{v}}\,ds-\int_{\cER}\delta u
{\partial_n \overline{v}}\,ds\nonumber\\&&
-\int_{\cER}\delta{\partial_n u}\overline{v}\,ds
-\frac{1}{i\kappa}\int_{\cER}\delta{\partial_n u}{\partial_n \overline{v}}\,ds+i\kappa\int_{\cED}\alpha u\overline{v}\,ds\nonumber\\&&
+i\kappa\int_{\cER}(1-\delta)u\overline{v}\,ds -\int_{\cED}
\left({\partial_n u}\overline{v}+u{\partial_n \overline{v}}\right)\,ds.\label{Ahalt}
\end{eqnarray}
Here $\nabla_h$ is the piecewise defined gradient and  $\partial_n u =\nabla_h u\cdot n$ element by element. In addition the right had side is given by
\[
F(v)=-\frac{1}{\eye\kappa}\int_{\cER}\delta g\,{\partial_n \overline{v}}\,ds+\int_{\cER} (1-\delta )g\overline{v}\,ds
\]
  By virtue of the regularity 
of the solution of (\ref{helm})-(\ref{imp}) noted above, it satisfies
\[
A_h(u,v)=F(v)\]
for all sufficiently smooth test functions $v$ (for example piecewise $H^2$ is sufficient).

Now suppose we wish to discretize the problem.
Let $V_h\subset \Pi_{K\in T_h}H^2(K)$ be a finite dimensional space.    If $V_h$ is chosen to consist of piecewise smooth solutions of (\ref{helm}), we have the standard TDG
and seek an approximate  $u_h\in V_h$ that satisfies
\[
A_h(u_h,v)=F(v)\quad\mbox{ for all }v\in V_h.
\]
For piecewise constant media, the space $V_h$ can be chosen in many ways.  One choice uses Bessel functions (good for conditioning but bad for computational speed because of the need for quadrature), another, more standard choice, uses plane waves (typically worse conditioned but easier to use since integrals can be computed in closed form)~\cite{HMP13I}.

However if $\epsilon(x,y)$ is non-constant on an element we cannot use simple solutions of the Helmholtz equation.  In this paper we assume that $\epsilon(x,y)$ is a smooth function on each element (but may be discontinuous between elements).  Then, as we shall shortly describe, the space $V_h$ can be constructed using ``Generalized Plane Waves'' (GPWs) that approximately satisfy the Helmholtz equation.  However we have been unable to prove convergence for the standard TDG method in this case, and instead add a stabilizing term to the sesquilinear form so define
\[
B_h(u,v)=A_h(u,v)+\frac{i}{\kappa^2}\int_\Omega\gamma (\Delta u+\kappa^2\epsilon u)\overline{(\Delta v+\kappa^2\epsilon v)}\,dA.
\]
Here $\gamma>0$ is a new penalty parameter that is a piecewise constant function of position on 
the  mesh. We note that the new term vanishes on elements with constant material coefficients allowing plane waves to be  used and their the method reduces to the standard TDG.  Now we seek $u_h\in V_h$ such that
\begin{equation}\label{Bhelm}
B_h(u_h,v)=F(v)\quad\forall v\in V_h.
\end{equation}

In the next section we describe how to construct GPWs element by element and hence complete the specification of the method.

\section{Generalized Plane Waves}\label{GPW}
In this section we focus specifically on GPWs. Firstly we describe the design process, including an explicit algorithm to build a local set of GPWs on a given element of the mesh $\mathcal T_h$. Secondly we turn to interpolation  of a solution of \eqref{helm} and prove error estimates. We provide various interpolation properties of such a set of local GPWs on a given element of the mesh, and derive a global interpolation property on the whole domain $\Omega$ by piecewise GPWs. Thirdly we prove a result on approximation of the space of bi-variate polynomials of degree 1 by GPWs. This result will be useful for the error analysis.

\subsection{Design and interpolation properties}
GPWs have been introduced in \cite{LM_thesis,IGD2013}. They generalize the use of classical plane waves, as exact solutions of an equation with piecewise constant coefficients, to the case of variable coefficients. The GPWs are not exact solutions of (\ref{helm}) but approximately solve the equation element by element.
 Their design process is based on a Taylor expansion and ensures that the homogeneous equation is locally satisfied up to a given order on each element $K$ of the mesh. 

On a given element $K$, consider the centroid $(x_{K},y_{K})$. A GPW on $K$ is a function $\varphi = e^{P}$ where
 \begin{equation}\label{eq:polnotation}
  P(x,y)=\sum_{i=0}^{{\rm{}d}_P} \sum_{j=0}^{{\rm{}d}_P-i} \lu{i}{j} \left(x-x_K\right)^i \left(y-y_K\right)^j,
  \end{equation}
  ${\rm{}d}_P$ being the total degree of the polynomial $P$. A GPW is designed to be an approximate solution of the Helmholtz equation: the polynomial coefficients $\left\{ \lu{i}{j},0\leq i+j\leq {\rm d}_P  \right\}$ are computed from the Taylor expansion of the variable coefficient $\epsilon$ in order for the function $\varphi = e^{P}$ to satisfy
\begin{equation}\label{eq:-lapl+al}
[ \Delta + \kappa^{2}\epsilon ] e^{P(x,y)} = O \left(\| (x,y)-(x_K,y_K) \|^q\right).
\end{equation}
The parameter $q$ is the order of approximation of the equation. Canceling all the terms of order less than $q$ in the Taylor expansion \eqref{eq:-lapl+al} is equivalent to a non linear system of $q(q+1)/2$ non linear equations. The unknowns of this system are the $({\rm{}d}_P+1)({\rm{}d}_P+2)/2$ coefficients of $P$. Setting simultaneously ${\rm{}d}_P = q+1$ and giving the values of the $2q+3$ coefficients $\left\{ \lu{i}{j}, i\in \{0,1\}, 0\leq j \leq q+1-i \right\}$ leads to a unique solution of the non linear system. This solution is explicitly expressed as
\begin{equation}\label{eq:IF}
\begin{array}{l}
\forall (i,j) \text{ s.t. } 0\leq i+j \leq q-1, \\
\displaystyle \lu{i+2}{j} 
= \frac1{(i+2)(i+1)}\Bigg(-\kappa^2\frac{\dux^i \duy^j \epsilon\left(x_K,y_K\right)}{i!j!} -  (j+2)(j+1) \lu{i}{j+2} 
\\ \displaystyle \phantom{\frac{\dux^i\beta\left(\mathbf{g}_K\right) \duy^j \beta\left(\mathbf{g}_K\right)}{i!j!} = } 
-\sum_{k=0}^{i} \sum_{l=0}^{j} (i-k+1)(k+1) \lu{i-k+1}{j-l} \lu{k+1}{l}
\\ \displaystyle \phantom{\frac{\dux^i\beta\left(\mathbf{g}_K\right) \duy^j \beta\left(\mathbf{g}_K\right)}{i!j!}  = } 
-\sum_{k=0}^{j} \sum_{l=0}^{i} (j-k+1)(k+1) \lu{i-l}{j-k+1} \lu{l}{k+1}\Bigg),
\end{array}
\end{equation} 
where $\partial_x=\partial/\partial_x$ and $\partial_y=\partial/\partial y$.
More precisely, as defined in \cite{IG2015}, a GPW at $(x_K,y_K)$ corresponds to the following normalization :
\begin{itemize}
\item[$\bullet$] $\lambda_{0,0}=0$,
\item[$\bullet$] $(\luoz,\luzo)=N (\cos \theta,\sin\theta)$, for some $N\in\mathbb C$ and $\theta\in\mathbb R$,
\item[$\bullet$] $ \lu{i}{j}=0$ for $i\in\{0,1\}$ and $1<i+j\leq q+1$.
\end{itemize}
  
A local set of linearly independent GPWs is then obtained for a given value of $N$ by considering $p$ equi-spaced directions $\theta_{l} = 2\pi(l-1)/p$ for $1\leq l \leq p$. The interpolation properties of this set of functions are the main topic of  \cite{IG2015}. The main result of that paper provides a sufficient condition on the parameters p and q to achieve a high order interpolation of a smooth solution of \eqref{helm} by GPWs, as well as a high order interpolation of its gradient. We will denote by $\GPWK$ the space spanned by the $p$ GPWs corresponding to $\theta_{l} = 2\pi(l-1)/p$ for $1\leq l \leq p$ and $N=\sqrt{ -\kappa^2 \epsilon(x_K,y_K)}$, or $N=\imath \kappa\sqrt{ \epsilon(x_K,y_K)}$. Let ${\mathcal C}^k(S)$ denote the set of functions with $k$ continuous derivatives on a set $S$. As a reminder, with the present notation, the interpolation result reads:
\begin{theorem}\label{th:u-ua}
Consider $K\in \mathcal T_h$ together with $n\in\mathbb N$ such that $n>0$. Assume that $q\geq n+1$, $p=2n+1$ and $\mathbf{g}_K=(x_K,y_K)\in K$ is the centroid of $K$. 
Finally suppose the coefficient $\epsilon \in\mathcal C^{q-1}(K)$.
Consider a solution $u$ of scalar wave equation \eqref{helm}, satisfying $u\in \mathcal C^{n+1}$. 
Then there is a function $u_a\in \GPWK$ implicitly depending on $\epsilon$ and its derivatives, and a constant $C(\kappa,K,n)$, implicitly depending on  $\epsilon$ and its derivatives as well, such that: for all $\mathbf{m}\in K$
\e{\label{eq:gradumua}
\left\{ 
\begin{array}{l}
\left| u\left(\mathbf{m}\right)-u_a\left(\mathbf{m}\right)\right| \leq C(\kappa,K,n) \left|\mathbf{m}-\mathbf{g}_K\right|^{n+1} \left\| u \right\|_{\mathcal C^{n+1}(K)} ,\\
\left\| \nabla u\left(\mathbf{m}\right)-\nabla u_a\left(\mathbf{m}\right)\right\| \leq C(\kappa,K,n) \left|\mathbf{m}-\mathbf{g}_K\right|^{n} \left\| u \right\|_{\mathcal C^{n+1}(K)}. 
\end{array}
\right.
}
\end{theorem} 
The following interpolation property of higher order derivatives stems directly from the proof of the previous theorem.
\begin{theorem}\label{th:u-ua2}
Consider $K\in \mathcal T_h$ together with $n\in\mathbb N$ such that $n>0$. Assume that $q\geq n+1$, $p=2n+1$ and $\mathbf{g}_K=(x_K,y_K)\in K$ is the centroid of $K$. 
Finally suppose the coefficient $\epsilon \in\mathcal C^{q-1}(K)$.
Consider a solution $u$ of scalar wave equation \eqref{helm}, satisfying $u\in \mathcal C^{n+1}$. 
Then the function $u_a\in \GPWK$ and the constant $C(\kappa,K,n)$ provided by Theorem \ref{th:u-ua} also satisfies : for all $\mathbf{m}\in K$ and all $j$ such that $0\leq j \leq k$
\e{\label{eq:derumua}
\left| \dx^{j}\dy^{k-j}u\left(\mathbf{m}\right)-\dx^{j}\dy^{k-j}u_a\left(\mathbf{m}\right)\right| \leq C(\kappa,K,n)\frac{(n+1)!}{(n+1-k)!} \left|\mathbf{m}-\mathbf{g}_K\right|^{n+1-k} \left\| u \right\|_{\mathcal C^{n+1}(K)} ,
}
where $k\leq n$.
Moreover there is a constant $\mathfrak C(\kappa,K,n)$ such that for all $\mathbf{m}\in K$ 
\e{\label{eq:opua}
\left| [\Delta +\kappa^2\epsilon]u_a\left(\mathbf{m}\right)\right| \leq \mathfrak C(\kappa,K,n) \left|\mathbf{m}-\mathbf{g}_K\right|^{n+1} \left\| u \right\|_{\mathcal C^{n+1}(K)} .
}
\end{theorem} 
\begin{proof}
The interpolation property of GPWs for any derivative of $u$ directly stems from the Taylor expansion of $u-u_{a}$, exactly as for the gradient and this proves \eqref{eq:derumua}.


The design of GPWs directly yields that for all $l$ such that $1\leq l \leq p$ the corresponding GPW satisfies
$$
\left| [\Delta +\kappa^2\epsilon]\varphi_{l} \right| \leq C_{l} \left|\mathbf{m}-\mathbf{g}_K\right|^{q }.
$$
Moreover  $\displaystyle u_a= \sum_{l=1}^{2n+1} \mathsf X_l \varphi_l$ and it was already noticed in \cite{IG2015} that  $|\mathsf X_l|\leq C(\kappa,K,n)\|u\|_{\mathcal C^{n+1}}$. As a result
$$
\left| [\Delta +\kappa^2\epsilon]u_a\left(\mathbf{m}\right)\right| 
=\left| \sum_{l=1}^{2n+1} \mathsf X_l [\Delta +\kappa^2\epsilon]\varphi_l \right|
\leq C(\kappa,K,n)\|u\|_{\mathcal C^{n+1}}\left|\mathbf{m}-\mathbf{g}_K\right|^{q } \sum_{l=1}^{2n+1} C_{l}
$$
and so \eqref{eq:opua} holds $\displaystyle \mathfrak C (\kappa,K,n) =C(\kappa,K,n) \sum_{l=1}^{2n+1} C_{l} $.
\end{proof}

The last step to build, from the local functions spaces $\GPWK$, a set of GPWs on the whole domain $\Omega$: the GPWs space $V_h$ is naturally defined as $ \prod_{K\in T_h}\GPWK$. Note that $p$ and $q$ can vary from element to element.

As a result, we have the following estimate for $[\Delta +\kappa^2\epsilon] (u-v_h)$, where $u$ is a smooth solution of Equation \eqref{helm}:
\begin{lemma}\label{helm-est}
Suppose that $u$ is a solution of scalar wave equation \eqref{helm} which belongs to $\mathcal C^{n+1}(\Omega)$. 
Then the function $v_h\in V_h=\prod_{K\in T_h}\GPWK$,  provided element by element by Theorem \ref{th:u-ua}, satisfies :
 there exists a constant $C$ independent of $h$ such that
\[
\Vert [\Delta +\kappa^2\epsilon] (u-v_h) \Vert_{L^2(\Omega)} \leq C\mbox{area}(\Omega)^{1/2}\left(\max_{K\in T_h} h_K\right)^q\Vert u\Vert_{{\mathcal C}^q(\Omega)}
\]
where $h_K$ is the radius of $K$.
\end{lemma}

\subsection{Approximation of linear polynomials}
The result \cite[Lemma 3.10]{git09} addresses the approximation of bi-variate polynomials of degree 1 by classical plane waves, and here we are interested in the approximation of bi-variate polynomials of degree 1 by GPWs. This result is needed to apply the $h$-based
analysis of \cite{git09} or \cite{kapita14}.
\begin{lemma}\label{linfun}
Consider $\hat K\in [0,1]^2$ the reference element. Suppose $n\in\mathbb N$ is such that $n\geq 2$. For $p=2n+1$ and $q\geq n+1$ there is a constant $C$ independent of $\kappa$ (but not of $p$) such that
$$
\inf_{v\in \GPWKhat} \|f-v\|_{0,\hat K}
\leq C\kappa^2|\epsilon(x_{\hat K},y_{\hat K})|\|f\|_{0,\hat K},\ \forall f\in \mathcal P_1(\mathbb R^2).
$$ 
\end{lemma}
\begin{remark}The proof strongly relies on the fact that the GPW space is designed with the normalization $(\luoz^k,\luzo^k) = \imath \kappa\sqrt{ \epsilon(x_{\hat K},y_{\hat K})}(\cos\theta_k,\sin\theta_k)$, for equi-spaced angles $\theta_k$, for $1\leq k \leq p$.\end{remark}
\begin{proof} For the sake of clarity we define $\tilde \kappa = \kappa\sqrt{ \epsilon(x_{\hat K},y_{\hat K})}$. Consider 
$$
b_j :=(i\tilde \kappa)^{-[j/2]}\sum_{k=1}^p \alpha_k^{(j)} \varphi_k 
$$
where the $\varphi_k$s are the GPWs, $\alpha_k^{(j)} = (\mathsf{M}_p)^{-1}$, $p=2m+1$ and $\mathsf{M}_p\in\mathbb R^{p,p}$ is defined for $1\leq k,l \leq p$ by
$$
(\mathsf M_p)_{kl} :=\left\{
\begin{array}{l}
1\text{ for }l = 1 
\\\displaystyle 
\cos\left(\frac{l}{2}\theta_k\right)\text{ for }l \text{ even}
\\\displaystyle 
\sin\left(\frac{l-1}{2}\theta_k\right)\text{ for }l \geq 3 \text{ odd}
\end{array} \right.
$$
We know that
$$
\sum_{k=1}^p \alpha_k^{(j)} \varphi_k =
\sum_{n=0}^\infty\frac{1}{n!}\sum_{k=1}^p \alpha_k^{(j)} (P_k(x,y))^n ,
$$
and, if $\hat{\mathbf x} = (\hat x,\hat y)=(x-x_{\hat K},y-y_{\hat K})$, then
$$
\begin{array}{rl}
(P_k(x,y))^n &\displaystyle=
\left(
 \begin{pmatrix}
\luoz^k\\ \luzo^k
\end{pmatrix}
\cdot
\begin{pmatrix}
\hat x \\ \hat y
\end{pmatrix}
+\tilde \kappa^2 f_{k,q}({\mathbf{x}})
\right)^n  \text{ see Lemma \ref{lem:Pk}}
\\&\displaystyle=
\left(
 \begin{pmatrix}
\luoz^k\\ \luzo^k
\end{pmatrix}
\cdot
\begin{pmatrix}
\hat x \\ \hat y
\end{pmatrix}
\right)^n
+\tilde \kappa^n\sum_{j=0}^{n-1}\begin{pmatrix}
n\\ j
\end{pmatrix}
\imath^j\tilde \kappa^{n-j}\left(
 \begin{pmatrix}
\luoz^k\\ \luzo^k
\end{pmatrix}
\cdot
\begin{pmatrix}
\hat x \\ \hat y
\end{pmatrix}
\right)^j
f_{k,q}({\mathbf{x}})^{n-j}
\\&\displaystyle=
\left(
 \begin{pmatrix}
\luoz^k\\ \luzo^k
\end{pmatrix}
\cdot
\begin{pmatrix}
\hat x \\ \hat y
\end{pmatrix}
\right)^n
+\tilde \kappa^{n+1}g_{k,q}({\mathbf{x}}),
\end{array}
$$
 the function $g_{k,q}$, defined as $g_{k,q}({\mathbf{x}}) = \sum_{j=0}^{n-1}\begin{pmatrix}
n\\ j
\end{pmatrix}
\imath^j\tilde \kappa^{n-1-j}\left(
 \begin{pmatrix}
\luoz^k\\ \luzo^k
\end{pmatrix}
\cdot
\begin{pmatrix}
\hat x \\ \hat y
\end{pmatrix}
\right)^j
f_{k,q}({\mathbf{x}})^{n-j}$, being uniformly bounded for $\mathbf x =(x,y)\in \hat K$ as $\tilde \kappa\rightarrow 0$.

We assume that $(\luoz^k,\luzo^k) = \imath \tilde \kappa(\cos\theta_k,\sin\theta_k)$. Define
$$
K_{ 0  }^n(\mathbf{x}) = \frac{1}{2\pi}\int_{-\pi}^{\pi} \left(\begin{pmatrix}
\luoz^k\\ \luzo^k
\end{pmatrix}
\cdot
\begin{pmatrix}
\hat x \\\hat  y
\end{pmatrix}\right)^n
d\theta \quad \text{ and }\forall 1\leq j\leq n:
$$
$$
K_{2j  }^n(\mathbf{x}) = \frac{1}{\pi}\int_{-\pi}^{\pi} \left(\begin{pmatrix}
\luoz^k\\ \luzo^k
\end{pmatrix}
\cdot
\begin{pmatrix}
\hat x \\ \hat y
\end{pmatrix}\right)^n \cos(j\theta)d\theta
\text{, }
K_{2j+1}^n (\mathbf{x})= \frac{1}{\pi}\int_{-\pi}^{\pi} \left(\begin{pmatrix}
\luoz^k\\ \luzo^k
\end{pmatrix}
\cdot
\begin{pmatrix}
\hat x \\ \hat y
\end{pmatrix}\right)^n \sin(j\theta)d\theta.
$$
The leading order term of $(P_k(x,y))^n$ as $\tilde \kappa\rightarrow 0$ is $\left(
 \begin{pmatrix}
\luoz^k\\ \luzo^k
\end{pmatrix}
\cdot
\begin{pmatrix}
\hat x \\ \hat y
\end{pmatrix}
\right)^n$. As shown in \cite{git09} , 
it can be written
$
\left(
 \begin{pmatrix}
\luoz^k\\ \luzo^k
\end{pmatrix}
\cdot
\begin{pmatrix}
\hat x \\ \hat y
\end{pmatrix}
\right)^n
=(\imath\tilde\kappa)^n\sum_{l=1}^{2n+1}K_l^n(\mathbf{x})\mathsf{M}_{lk}
$ so that
$$
\begin{array}{l}\displaystyle
\sum_{n=0}^\infty\frac{1}{n!}\sum_{k=1}^p \alpha_k^{(j)}
\left(
 \begin{pmatrix}
\luoz^k\\ \luzo^k
\end{pmatrix}
\cdot
\begin{pmatrix}
\hat x \\ \hat y
\end{pmatrix}
\right) ^n \\\displaystyle
=
\sum_{n=0}^\infty\frac{(\imath\tilde \kappa)^n}{n!}
\sum_{l=1}^{2n+1}K_l^n(\mathbf{x})
\sum_{k=1}^p \alpha_k^{(j)}\mathsf{M}_{lk}\\\displaystyle
=
\sum_{n=0}^m\frac{(\imath\tilde \kappa)^n}{n!}
\sum_{l=1}^{2n+1}K_l^n(\mathbf{x})
\sum_{k=1}^p \alpha_k^{(j)}\mathsf{M}_{lk}+
\sum_{n=m+1}^\infty\frac{(\imath\tilde \kappa)^n}{n!}
\sum_{l=1}^{2n+1}K_l^n(\mathbf{x})
\sum_{k=1}^p \alpha_k^{(j)}\mathsf{M}_{lk},
\end{array}
$$
and since $K_j^n=0$ if $[j/2]>n$ and $\sum_{k=1}^p \alpha_k^{(j)}\mathsf{M}_{lk}=\delta_{jl}$ for $1\leq l,j\leq p$, it yields
$$
\begin{array}{rl}\displaystyle
\sum_{n=0}^\infty\frac{1}{n!}\sum_{k=1}^p \alpha_k^{(j)}
\left(
 \begin{pmatrix}
\luoz^k\\ \luzo^k
\end{pmatrix}
\cdot
\begin{pmatrix}
\hat x \\ \hat y
\end{pmatrix}
\right) ^n 
&\displaystyle
=
\sum_{n=[j/2]}^m\frac{(\imath\tilde \kappa)^n}{n!}\left(K_j^n(\mathbf{x})+
\sum_{l=p+1}^{2n+1}K_l^n(\mathbf{x})
\sum_{k=1}^p \alpha_k^{(j)}\mathsf{M}_{lk}\right)\\&\displaystyle+
\sum_{n=m+1}^\infty\frac{(\imath\tilde \kappa)^n}{n!}
\left(K_j^n(\mathbf{x})+
\sum_{l=p+1}^{2n+1}K_l^n(\mathbf{x})
\sum_{k=1}^p \alpha_k^{(j)}\mathsf{M}_{lk}\right)\\
&\displaystyle
=
\sum_{n=[j/2]}^m\frac{(\imath\tilde \kappa)^n}{n!}K_j^n(\mathbf{x})+
\tilde \kappa^{m+1}R_j(\tilde \kappa,\mathbf{x})
\end{array}
$$
since $l/2\geq m+1\Rightarrow [K_l^n = 0$ for $n\leq m]$, and where the remainder function $R_j$ defined by 
$$
R_j(\kappa,\mathbf{x})
= \frac{1}{\tilde \kappa^{m+1}}
\sum_{n=m+1}^\infty\frac{(\imath\tilde \kappa)^n}{n!}
\left(K_j^n(\mathbf{x})+
\sum_{l=p+1}^{2n+1}K_l^n(\mathbf{x})
\sum_{k=1}^p \alpha_k^{(j)}\mathsf{M}_{lk}\right)
$$
is uniformly bounded on $\hat K$. So

$$
\sum_{k=1}^p \alpha_k^{(j)} \varphi_k =
\sum_{n=[j/2]}^m
\left(
\frac{(\imath\tilde \kappa)^n}{n!}K_j^n(\mathbf{x})
+\tilde \kappa^{n+1}g_{k,q}(\mathbf{x}) 
\right)
+\tilde \kappa^{m+1}R_j(\tilde \kappa,\mathbf{x})
$$
Since  $b_j(\mathbf{x}) =(i\tilde \kappa)^{-[j/2]}\sum_{k=1}^p \alpha_k^{(j)} \varphi_k  $ it clearly shows that
$$
\lim_{\tilde \kappa\rightarrow 0} b_j(\mathbf{x})
=
\frac{1}{{[j/2]}!}K_j^{[j/2]}(\mathbf{x}).
$$
As a consequence, the definition of $K_j^{[j/2]}$ combined with $((\luoz^k,\luzo^k) = \imath \tilde \kappa(\cos\theta_k,\sin\theta_k))$ lead to
$$
b_1(\mathbf{x}) = 1 + O(\tilde \kappa^2),\ 
b_2(\mathbf{x}) = x-x_K + O(\tilde \kappa^2),\
b_3(\mathbf{x}) = y-y_K + O(\tilde \kappa^2).
$$
\end{proof}

To complete the proof of Lemma \ref{linfun}, we need to prove the following result.
\begin{lemma}\label{lem:Pk}
 Suppose $n\in\mathbb N$ is such that $n\geq 2$. For $p=2n+1$ and $q\geq n+1$, consider the basis of $p$ functions $\varphi_k\in \GPWKhat$ approximating \eqref{helm} at the point  $\mathbf g_{\hat K}=(x_{\hat K},y_{\hat K})$ at order $q$, and, for all $k$ such that $1\leq k\leq p$, the corresponding polynomials $\displaystyle P_k(x,y)=\sum_{0\leq i+j\leq q+1} \luij^k (x-x_{\hat K})^i (y-y_{\hat K})^j$ satisfying $\varphi_k=\exp P_k$. These polynomials satisfy
$$
P_k(x,y)=
 \begin{pmatrix}
\luoz^k\\ \luzo^k
\end{pmatrix}
\cdot
\begin{pmatrix}
\hat x \\ \hat y
\end{pmatrix}
+\kappa^2 \epsilon(x_K,y_K) f_{k,q}({\mathbf{x}})
$$
where $\mathbf x = (x,y)$, $(\hat x,\hat y)=(x-x_{\hat K},y-y_{\hat K})$ and the remainder function $f_{k,q}$ is uniformly bounded on $\hat K$.
\end{lemma}
\begin{proof}
The normalization $(\luoz^k,\luzo^k) = \imath \kappa\sqrt{ \epsilon(x_K,y_K)}(\cos\theta_k,\sin\theta_k)$ implies that  $\lu{2}{0}^k=0$ so that the induction formula \eqref{eq:IF} reads:
\begin{equation}\left\{
\begin{array}{cl}
\lu{2}{j}^k &\displaystyle
 = -\frac{1}{2}\frac{\kappa^2\partial_y^j\epsilon(x_K,y_K)}{j!} \\
\lu{3}{j}^k &\displaystyle
 = -\frac{1}{6}\left( \frac{\kappa^2\partial_x\partial_y^j\epsilon(x_K,y_K)}{j!} + 4\lu{2}{j}^k\luoz^k\right)\\
\lu{i+2}{j}^k &\displaystyle
 = -\frac{1}{(i+2)(i+1)}\Bigg( \frac{\kappa^2\partial_x^i\partial_y^j\epsilon(x_K,y_K)}{i!j!} + (j+2)(j+1)\lu{i}{j+2}^k\luoz^k \qquad \forall i>1,\\
 &\displaystyle \phantom = 
 + \sum_{k=1}^{i-1}\sum_{l=0}^{j}
 (i-k+1)(k+1)\lu{i-k+1}{j-l}^k\lu{k+1}{l}^k\\
 &\displaystyle \phantom = 
 + \sum_{l=1}^{i-1}\sum_{k=0}^{j}
 (j-k+1)(k+1)\lu{i-l}{j-k+1}^k\lu{l}{k+1}^k\Bigg)
\end{array}\right.
\end{equation}
This clearly completes the proof by induction.
\end{proof}

More precisely, the result needed in the following result that  corresponds to Lemma 3.12 from \cite{git09}. We state it here to specify the GPWs parameters, and provide no more than a sketch of the proof since it relies on Lemma \ref{linfun} but not specifically on the basis function set.
\begin{corollary} \label{lincor} Suppose that $n\in\mathbb N$ is such that $n\geq 2$. For $p=2n+1$ and $q\geq n+1$, suppose $w_h^c$ is a linear function on a triangle $K$ then there is a GPW function $w_h\in \GPWK$ such that
\[
\Vert w_h^c-w_h\Vert_{L^2(K)}\leq Ch_K^2\Vert w_h^c\Vert_{L^2(K)}.
\]
\end{corollary}
\begin{proof}
 By translation and dilation by $1/h_K$ we can map an element $K$ to an element $\tilde{K}\subset \hat{K}= (0,1)^2$. Let $\hat{w}^c$ denote the transformed polynomial and note that $\hat{w}_h^c\in \mathcal P_1(\mathbb R^2)$. Let $\hat P$ the $L^2(\hat K)$-projection onto the plane wave space $GPW_{\hat{\kappa}}^{p,q}(\hat{K})$ where $\hat{\kappa} = h_K \kappa$. Applying  Lemma \ref{linfun}, we get
{
\begin{equation}\label{eq:est}
\| (I-\hat P )\hat{w}_h^c\|_{0,\hat K} 
\leq C\hat{\kappa}^2|\epsilon(x_{\hat K},y_{\hat K})| \| \hat{w}_h^c \|_{0,\hat K}.
\end{equation}}
The conclusion then follows by transforming back to $K$ using the fact that the transformation from a triangle $K$ to the reference triangle $\tilde K$ changes the frequency into $\hat \kappa = h_K \kappa$  {
and the bound  $\|            \hat{w}_h^c \|_{0,\hat K}\leq C \|            \hat{w}_h^c \|_{0,\tilde K} $ with $C$ independent of $\tilde{K}$ which holds
because $ \hat{w}_h^c$ is a linear polynomial and the mesh is regular.}
\end{proof}

\section{Error Estimates}\label{EE}
In this section we start by establishing  a straightforward error estimate using the coercivity and boundedness of the sesquilinear form $B_h(\cdot,\cdot)$, and
then prove convergence in the global $L^2$ norm.

We define the obvious modification of the DG norm from \cite{hmp13} by
\begin{eqnarray*}
\Vert u\Vert_{DG}^2&=&
\frac{1}{\kappa}\Vert \beta^{1/2}\jmp{\nabla_h u}\Vert_{L^2(\cEI)}^2+\kappa\Vert \alpha^{1/2}\jmp{ u}\Vert_{L^2(\cEI)}^2+\frac{1}{\kappa}\Vert \delta^{1/2}\partial_nu\Vert_{L^2(\cER)}^2\\&&\quad+\kappa\Vert (1-\delta)^{1/2}u\Vert_{L^2(\cER)}^2+\kappa\Vert \alpha^{1/2}u\Vert_{L^2(\cED)}^2+\frac{1}{\kappa^2}\Vert \gamma^{1/2}(\Delta_h u+\kappa^2\epsilon u)\Vert_{L^2(\Omega)}^2
\end{eqnarray*}
where $\Delta_h$ is the Laplacian defined piecewise element by element.  Obviously this is a semi-norm, but due to the new term is also a norm even on functions that do not
exactly satisfy the Helmholtz equation.
\begin{lemma}
The semi-norm $\Vert \cdot\Vert_{DG}$ is a norm.
\end{lemma}
\begin{proof} Suppose $\Vert u\Vert_{DG}=0$ then
$u$ satisfies $\Delta u+\kappa^2\epsilon u=0$ element-wise, and the normal derivatives and function values have no jump across interior edges. So
$\Delta u+\kappa^2\epsilon u=0$ in $\Omega$.  In addition the Cauchy data vanishes and so $u=0$ in $\Omega$.  Hence $\Vert \cdot\Vert_{DG}$ is a norm.
\end{proof}

We also need a new DG+ norm:
\begin{eqnarray*}
\Vert u\Vert_{DG+}^2&=&\Vert u\Vert_{DG}^2+\kappa\Vert \beta^{-1/2}\avg{u}\Vert_{L^2(\cEI)}^2+\frac{1}{\kappa}\Vert \alpha^{-1/2}\avg{\nabla_h u}\Vert_{L^2(\cEI)}^2+
\kappa\Vert\delta^{-1/2} u\Vert^2_{L^2(\cER)}\\&&\quad+\frac{1}{\kappa}\Vert \alpha^{-1}\partial_n u\Vert^2_{L^2(\cED)}+\kappa^2\Vert\gamma^{-1/2} u\Vert_{L^2(\Omega)}.
\end{eqnarray*}
The following estimates hold:
\begin{lemma} Under the  assumption that $\alpha>0$, $\beta>0$, $1>\delta>0$ and $\gamma>0$ in the generalized TDG, and provided $u$ is such that $\Vert u\Vert_{DG}$ is finite,
\[
\Im(B_h(u,u))\geq \Vert u\Vert_{DG}^2.
\]
Provided  $\Vert u\Vert_{DG+}$ and $\Vert v\Vert_{DG}$ are finite,  there exists a constant $C$ independent of $\kappa$, $u$ and $v$ such that
\[
|B_h(u,v)|\leq C \Vert u\Vert_{DG+}\Vert v\Vert_{DG}.
\]
\end{lemma}
\begin{proof}
The coercivity estimate follows from in the usual way by considering $\Im (B_h(u,u))$ and using the  assumption that $\epsilon$ is real \cite{HMP11}.

To obtain the desired continuity, we integrate the term $\nabla_h u\cdot\nabla_h v$ term in the definition of $A_h(\cdot,\cdot)$ 
by parts to obtain, for any $u,v\in V_h$,
\begin{eqnarray}
A_h(u,v)&=&-\int_\Omega u(\overline{\Delta_h v+\kappa^2\epsilon v})\,dA+\int_{\cEI}
\avg{u}\jmp{\nabla_h \overline{v}}\,ds-\int_{\cEI}\avg{\nabla_h u}\jmp{\overline{v}}\,ds\nonumber
\\&&
+\int_{\cER}(1-\delta) u{\partial_n \overline{v}}\,ds-\frac{1}{\eye\kappa}\int_{\cEI}\beta\jmp{\nabla_h u}\jmp{\nabla_h \overline{v}}\,ds\nonumber\\
&&+\eye\kappa\int_{\cEI}\alpha\jmp{u}\jmp{\overline{v}}\,ds-\frac{1}{\eye\kappa}\int_{\cER}\delta{\partial_n u}{\partial_n \overline{v}} \,ds\nonumber\\
&&+\eye\kappa\int_{\cER}(1-\delta)u\overline{v}\,ds-\int_{\cER}\delta {\partial_n u}\,\overline{v}\,ds\label{UWVF}\\
&&-\int_{\cED}{\partial_n u}\overline{v}+i\kappa\int_{\cED}\alpha u\overline{v}.
\nonumber
\end{eqnarray}
The result now follows from the definition of $B_h(u,v)$ and the Cauchy-Schwarz inequality.
\end{proof}
The following result is now a standard consequence of the above estimates~\cite{MelenkEsterhazy12,hmp13}:
\begin{lemma}
There is a unique solution $u_h\in V_h$ that satisfies (\ref{Bhelm}), and 
the following error estimate holds with constant $C$ independent of $\kappa$, $u$, and $u_h$:
\begin{equation}
\Vert u-u_h\Vert_{DG}\leq C \Vert u-v\Vert_{DG+}\quad \mbox{ for all }v\in V_h.
\label{ceatype}
\end{equation}
\end{lemma}

To obtain an order estimate, we need to make specific choices of the parameters
$\alpha$, $\beta$, $\delta$ and $\gamma$. There are several choices in the literature
depending on the precise setting of the problem (see for example \cite{buf07,git09,hmp13,hmp15}).  In this paper we shall make the classical UWVF choice \cite{buf07}:
\begin{equation}
\alpha=\beta=\gamma=\delta=1/2,\label{p_uwvf}
\end{equation} 
so that we can use results from~\cite{kapita14}.  Then we choose for $\gamma$
\[
\gamma=\gamma_0 h_K^r
\]
where $\gamma_0$ is constant and $ r\geq 0$.  We shall examine the role of $r$ later.

Using the estimates from Section~\ref{GPW} we can then prove the following error estimate
\begin{theorem}\label{th:u-uh}  Suppose $n\in\mathbb N$ is such that $n\geq 2$ and consider $p=2n+1$ and $q\geq n+1$. Suppose $V_h$ is formed from $q$th order GPWs element by element using $p$ directions. Then
the solution $u_h\in V_h$ of (\ref{Bhelm}) exists for all $h>0$ independent of $\kappa$ and it satisfies the following estimate with constant $C$ independent of $\kappa$, $u$, and $u_h$:
\[
\Vert u-u_h\Vert_{DG}\leq C (h^{n-1/2}+h^{q+r/2}).
\]
Here $C$ depends on the $\Vert u\Vert_{{\cal C}^{\max(n+1,q)}(\Omega)}$ norm of $u$.
\end{theorem}
\begin{remark} Since we need $q\geq n+1$ in the GPW theory, we see that the choice $q=n+1$ guarantees that the approximation of
the Helmholtz equation is high enough order. \end{remark}

\begin{proof} We pick $v\in V_h$ in equation (\ref{ceatype}) element by element to be the approximation by GPWs denoted by $u_a$ in Theorem \ref{th:u-ua2}.  We now need to estimate each term in
$\Vert u-v\Vert_{DG+}$.
Using Lemma~\ref{helm-est} the new term
\begin{equation}\label{Luest}
\Vert \gamma^{1/2}(\Delta_h (u-v)+\kappa^2\epsilon (u-v)\Vert_{L^2(\Omega)}
=
\Vert \gamma^{1/2}(\Delta_h v+\kappa^2\epsilon v)\Vert_{L^2(\Omega)}
\leq Ch^{q+r/2} \Vert u\Vert_{{\cal C}^q}.
\end{equation}
The remaining terms can be estimated in using Theorem \ref{th:u-ua2}.  For example
\begin{eqnarray*}
\Vert \alpha^{-1/2}\avg{\nabla_h (u-v)}\Vert^2_{L^2(\cEI)}&\leq &C
\sum_K\Vert \alpha^{-1/2}\nabla_h (u-v)\Vert_{L^2(\partial K)}^2\\
&\leq&C\sum_K \max_{e\in\partial K}\alpha^{-1}(e)\Vert \nabla_h (u-v)\Vert_{L^2(\partial K)}^2\\
&\leq&C\sum_K \max_{e\in\partial K}\alpha^{-1}(e)\left[h_K^{-1}\Vert \nabla_h (u-v)\Vert_{L^2(K)}^2+h_K\Vert\nabla\nabla (u-v)\Vert_{L^2(K)}^2\right]
\end{eqnarray*}
where we have used the standard trace estimate on $\partial K$.  Using Theorem
\ref{th:u-ua2} with $k=2$ and Theorem \ref{th:u-ua} we obtain
\begin{eqnarray*}
\Vert \alpha^{-1/2}\avg{\nabla_h (u-v)}\Vert^2_{L^2(\cEI)}&\leq &C\sum_K \max_{e\in\partial K}\alpha^{-1}(e)h^{2n-1}_K h_K^2\Vert u\Vert^2_{{\cal C}^{n+1}(K)}\\&\leq& 
C\max_{e\in \cEI}\alpha^{-1}(e)h^{2n-1}\Vert u\Vert_{{\cal C}^{n+1}(\Omega)}^2.
\end{eqnarray*}
Of course under our assumptions $\alpha=1/2$.
The remaining terms are estimated in the same way.
\end{proof}
We now use the standard duality approach to prove an $L^2(\Omega)$ norm estimate on the error. 
\begin{theorem}\label{th:L2cv}
Suppose we choose $r=3$ in the penalty parameter $\gamma$,  $p=2n+1$, $n\geq 2$ and $q=n+1$.  Then 
there exists a constant $C$ depending on $\kappa$ but independent of $h$ such that
\[
\Vert u-u_h\Vert_{L^2(\Omega)}\leq C h^{s} \Vert u-u_h\Vert_{DG}
\]
for some $s$ with  $0<s<1/2$ depending on $\Omega$ (given in \cite[Theorem 2.3]{hmp13}).
\end{theorem}
Under best possible conditions we then have the following convergence estimate:
\begin{corollary}\label{uuhcor} Suppose $u$ is a smooth solution of the Helmholtz equation, that $r=3$, $p=2n+1$, $n\geq 2$ and $q= n+1$ then
\[
\Vert u-u_h\Vert_{L^2(\Omega)}\leq C h^{n+s-1/2}
\]
\end{corollary}
\begin{remark}  Since $s\leq 1/2$ the maximum rate of convergence predicted for the method assuming a smooth solution and best regularity is
$O(h^n)$.
\end{remark}
\begin{proof}
 Define the dual variable $z\in H^1(\Omega)$ to satisfy
\begin{eqnarray*}
\Delta z+\kappa^2\epsilon z&=& u-u_h\mbox{ in }\Omega,\\
{\partial_n z}-i\kappa z&=&0\mbox{ on }\Gamma_R,\\
 z&=&0\mbox{ on }\Gamma_D.
\end{eqnarray*}
Under the assumptions on the domain, it is easy to see that  $z\in H^{3/2+s}(\Omega)$, $s>1/2$, \cite{hmp13} is sufficiently regular  that
\[
A_h(\xi,z)=\int_\Omega \xi\overline{(u-u_h)}\,dA
\]
for all test function $\xi$ that are $H^2$ piecewise smooth. This follows from  (\ref{UWVF}).  Hence, by the definition of $B_h$ 
\[
B_h(\xi,z)=\frac{1}{\kappa^2}\int_\Omega \gamma (\Delta_h\xi+\kappa^2\epsilon \xi)\overline{(u-u_h)}\,dA+\int_\Omega \xi\overline{(u-u_h)}\,dA
\]
so choosing $\xi=u-u_h$ and letting $z_h\in V_h$ be arbitrary
\[
\Vert u-u_h\Vert_{L^2(\Omega)}^2=B(u-u_h,z-z_h)-\frac{1}{\kappa^2}\int_\Omega \gamma (\Delta_h(u-u_h)+\kappa^2\epsilon (u-u_h))\overline{(u-u_h)}\,dA.
\]
The second term on the right hand side can be estimated using the Cauchy-Schwarz and arithmetic-geometric mean  inequality to give 
\begin{eqnarray*}
\left|\int_\Omega \gamma (\Delta_h(u-u_h)+\kappa^2\epsilon (u-u_h))\overline{(u-u_h)}\,dA\right|
&\leq &\frac{1}{k}\Vert u-u_h\Vert_{DG}\Vert \gamma^{1/2} (u-u_h)\Vert_{L^2(\Omega)}\\&\leq& \frac{\gamma_{max}}{2\kappa^2}\Vert u-u_h\Vert_{DG}^2+\frac{1}{2}\Vert u-u_h\Vert_{L^2(\Omega)}^2,
\end{eqnarray*}
where $\gamma_{max}=\max_{x\in\Omega}\gamma=O(h^r)$.

To estimate $B_h(u-u_h,z-z_h)$ we integrate the grad-grad term in $A_h(u,v)$ by parts onto $u$ to obtain
\begin{eqnarray}
A_h(u,v)&=& -\int_{\Omega}(\Delta_h u+\kappa^2\epsilon u)\overline v\,dA+\int_{\cEI}\left(\jmp{\nabla_h u}\avg{\overline{v}}
-\jmp{ u}\cdot\avg{\nabla_h\overline{v}}\right)\,ds-\frac{1}{i\kappa}\int_{\cEI}\beta\jmp{\nabla_h u}\jmp{\nabla_h\overline{v}}\,ds\nonumber\\&&
+{i\kappa}\int_{\cEI}\alpha\jmp{ u}\cdot\jmp{\overline{v}}\,ds-\int_{\cER}\frac{\delta}{i\kappa} (i\kappa u-{\partial_n u}){\partial_n \overline{v}}\,ds\nonumber\\&&
+\int_{\cER}(1-\delta)({\partial_n u}-i\kappa u)\overline{v}\,ds+\int_{\cED}u(i\kappa\alpha \overline{v}-\partial_n\overline{v})\,ds.
\label{Ahap}\end{eqnarray}
Using this in the definition of $B_h(u,v)$ shows that
\[
|B_h(u,v)|\leq C\Vert u\Vert_{DG}\Vert v\Vert_{DG+} 
\]
where $C$ is independent of  $u$ and $v$ so that we have the estimate
\begin{equation}
\Vert u-u_h\Vert_{L^2(\Omega)}^2\leq C \Vert u-u_h\Vert_{DG}\Vert z-z_h\Vert_{DG+} +\frac{\gamma_{max}}{2\kappa^2}\Vert u-u_h\Vert_{DG}^2.
\label{eqest}
\end{equation}
It is now necessary to choose $z_h$.  Following the proof of \cite[Theorem 5.6]{kapita14}, let $z_h^c$ denote the continuous piecewise linear finite element interpolant of $z$. We choose $z_h\in V_h$ to be the GPW approximation of $z_h^c$  constructed in
Lemma~\ref{lincor}).   Of course
\[
\Vert z-z_h\Vert_{DG+}\leq \Vert z-z^c_h\Vert_{DG+}+\Vert z_h^c-z_h\Vert_{DG+}
\]
and it remains to estimate each term. Estimates from the proof of \cite[Theorem 5.6]{kapita14} show that on each interior edge in the mesh
\begin{eqnarray*}
\Vert\alpha^{-1/2}\avg{\nabla_h(z-z^c_{h})}\Vert_{L^2(e)}
&\leq& C\sum_{j=1}^2h_{K_j}^{s}\vert z\vert_{H^{3/2+s}(K_j)}\\
\Vert\beta^{-1/2}\avg{z-z^c_{h}}\Vert_{L^2(e)}&\leq&C\sum_{j=1}^2h^{1+s}_{K_j}\vert
z\vert_{H^{3/2+s}(K_j)},
\end{eqnarray*}
with corresponding entries results for the jumps in the above quantities and for boundary terms.  In addition
\begin{eqnarray*}
\Vert \gamma^{1/2}(\Delta_h (z-z_h^c)+\kappa^2\epsilon (z-z_h^c)\Vert_{L^2(\Omega)}
&=&
\Vert \gamma^{1/2}(u-u_h-(\Delta_h z_h^c+\kappa^2\epsilon z_h^c))\Vert_{L^2(\Omega)}\\
&\leq& Ch^{r/2}\Vert u-u_h\Vert_{L^2(\Omega)}+ \Vert \gamma^{1/2}(\Delta_h z_h^c+\kappa^2\epsilon z_h^c)\Vert_{L^2(\Omega)}.
\end{eqnarray*}
On an element $K$ we can use the regularity of the mesh to establish local inverse estimates and prove:
\begin{eqnarray*}
&& \Vert \gamma^{1/2}\Delta_h z_h^c\Vert_{L^2(K)}\leq Ch_K^{r/2-1}\Vert z_h^c\Vert_{H^1(K)}\\
& \leq &Ch_K^{r/2-1}(\Vert z_h^c-z\Vert_{H^1(K)}+\Vert z\Vert_{H^1(K)})\leq Ch_K^{r/2-1}\Vert z\Vert_{H^{3/2+s}(K)}.
 \end{eqnarray*}
Proceeding similarly for the lower order term, we conclude that provided $r/2>1$ we have 
\[
\Vert \gamma^{1/2}(\Delta_h (z-z_h^c)+\kappa^2\epsilon (z-z_h^c))\Vert_{L^2(\Omega)}\leq Ch^{r/2-1}\Vert z\Vert_{H^{3/2+s}(\Omega)}.
\]
In addition we must estimate
\begin{eqnarray*}
\Vert \gamma^{-1/2}(z-z_h^c)\Vert^2_{L^2(\Omega)}&=&\sum_K\int_K\gamma^{-1}(z-z_h^c)^2\,dA\leq C\sum_Kh_K^{3+2s-r}\Vert z\Vert_{H^{3/2+s}(K)}^2\\
&\leq& Ch^{3+2s-r}\Vert z\Vert_{H^{3/2+s}(\Omega)}^2.
\end{eqnarray*}
Taken together, if $3+2s\geq r\geq 2$ we have
\[
 \Vert z-z^c_h\Vert_{DG+}\leq C(h^{r/2-1}+h^{3/2+s-r/2})\Vert z\Vert_{H^{3/2+s}(\Omega)}
 \]
 A good choice is then $r=3$ since in that case $r/2-1\geq s$ and using the a priori estimate for $z$ from \cite[Theorem 2.3]{hmp13}
 \[
 \Vert z-z^c_h\Vert_{DG+}\leq Ch^s\Vert u-u_h\Vert_{L^2(\Omega)}.
 \]

It now remains to estimate $\Vert z_h^c-z_h\Vert_{DG+}$. As we have seen there are two troublesome terms:
$\Vert \gamma^{-1/2}(z_h^c-z_h)\Vert^2_{L^2(\Omega)}$ and $\Vert \gamma^{1/2}(\Delta_h (z^c_h-z_h)+\kappa^2\epsilon (z_h^c-z_h))\Vert_{L^2(\Omega)}$
with the remaining terms following using Lemma \ref{linfun} as in \cite{kapita14}.  Using first a local inverse estimate, then Lemma \ref{linfun},
\[
\Vert \gamma^{1/2}(\Delta_h (z^c_h-z_h)+\kappa^2\epsilon (z_h^c-z_h))\Vert_{L^2(K)}\leq Ch_K^{r/2-2}\Vert z_h^c-z_h\Vert_{L^2(K)}
\leq Ch_K^{r/2}\Vert z_h^c\Vert_{L^2(K)}
\]
so that, squaring and adding, and using the a priori estimate for $z$ from \cite[Theorem 2.3]{hmp13}
\begin{eqnarray*}
\Vert \gamma^{1/2}(\Delta_h (z^c_h-z_h)+\kappa^2\epsilon (z_h^c-z_h))\Vert_{L^2(\Omega)}&\leq &Ch_K^{r/2}\Vert z_h^c\Vert_{L^2(\Omega)}\\
&\leq& Ch_K^{r/2}(\Vert z-z_h^c\Vert_{L^2(K)}+\Vert z\Vert_{L^2(K)})\\&\leq& Ch^{r/2}\Vert u-u_h\Vert_{L^2(\Omega)}.
\end{eqnarray*}
To estimate the global $L^2$ term, again using Lemma \ref{linfun}, 
\[
\Vert \gamma^{-1/2}(z_h^c-z_h)\Vert^2_{L^2(K)}\leq Ch_K^{-r/2} \Vert (z_h^c-z_h)\Vert^2_{L^2(K)}
\leq Ch_K^{2-r/2}  \Vert z_h^c\Vert^2_{L^2(K)}
\]
Adding over all elements and using the  a priori estimate for $z$ from \cite[Theorem 2.3]{hmp13}
\[
\Vert \gamma^{-1/2}(z_h^c-z_h)\Vert^2_{L^2(\Omega)}\leq  Ch^{2-r/2}  \Vert u-u_h\Vert^2_{L^2(K)}
\]

We have thus proved that when $r=3$ and noting $0<s<1/2$ we have 
\[
\Vert z-z_h\Vert\leq Ch^{s}\Vert u-u_h\Vert_{L^2(\Omega)}.
\]
so we conclude the desired result using this result in (\ref{eqest}).\end{proof}
\section{Numerical Tests}\label{NT}
We now test the GPW based RDG method on two test problems with a known solution: Airy waves (linear variation in $\epsilon$) and Weber Waves (quadratic variation in $\epsilon$).  {In Figs.~\ref{ghqs}, \ref{ghcomb}, \ref{Wcomb} we plot the relative $L^2(\Omega)$ error in the computed solution against a parameter labeled $C/h$.  This is computed using the total number of degrees of freedom $N_{\rm{}dof}$ and the number of directions per element $p=2n+1$  as $\sqrt{N_{\rm{}dof}/p}$.  We choose this parameter since, in our theorems, convergence is in terms of mesh size rather than total number of degrees of freedom.}

\subsection{Airy Waves.}\label{Airy}
The simplest example of a spatially dependent refractive index is $\epsilon(x,y)=-y$ on the domain $[-1,1]\times[-1,1]$.  We can then choose Dirichlet boundary data (for our theory we need an impedance boundary condition, but the same result holds in the case provided $\kappa$ is not an
eigenvalue of the domain) such that the exact solution is
\[
u(x,y)=Ai(\kappa^{2/3}y)
\]
where $Ai(r)$ is the Airy function as shown in the left panel of Fig.~\ref{FL1}.
\begin{figure}
\begin{center}
\begin{tabular}{cc}
\resizebox{0.45\textwidth}{!}{\includegraphics{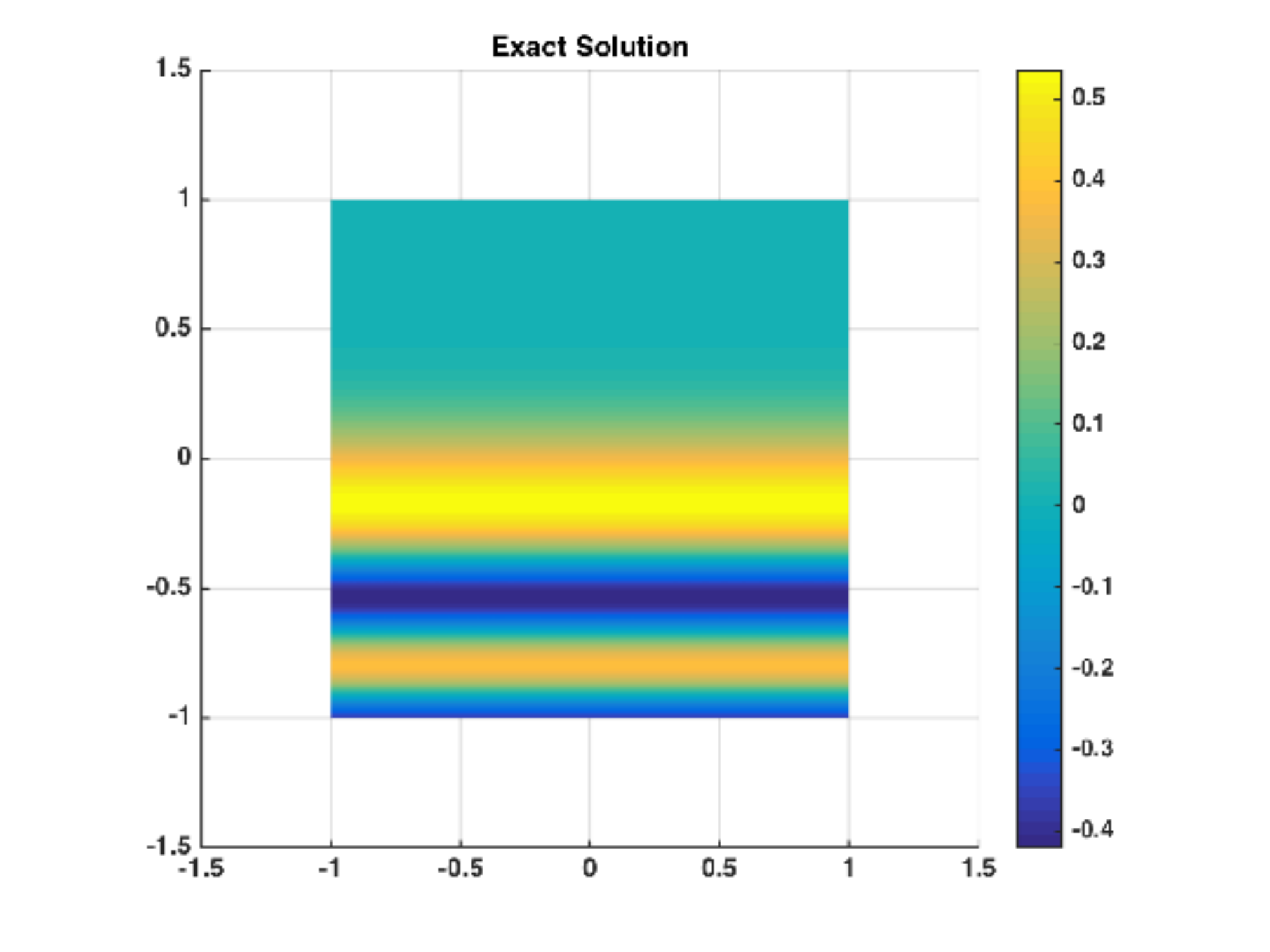}}
\resizebox{0.35\textwidth}{!}{\includegraphics{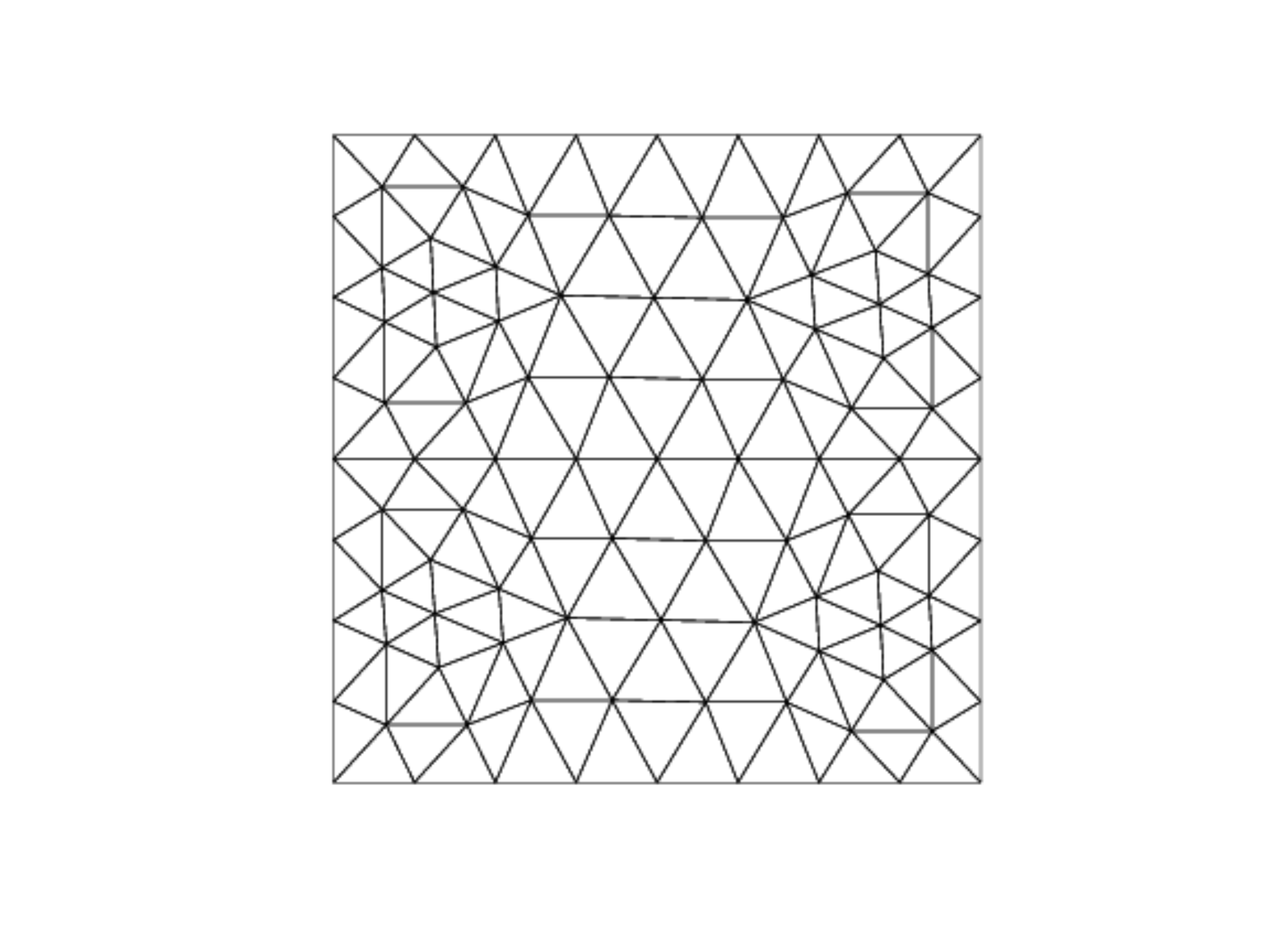}}
\end{tabular}
\end{center}
\caption{Left: Exact Airy function solution.  Right: Initial mesh.}
\label{FL1}
\end{figure}


This solution is oscillatory for $y<0$ and exponentially decaying for $y>0$. In all the experiments, we make the choice
\[
\alpha=\beta=\delta=1/2,
\]
The initial mesh for the experiments is shown Fig.~\ref{FL1} right panel. 

\subsubsection{The case $\gamma=h^3$}
{
Starting with the mesh in Fig.~\ref{FL1} and using uniform refinement, we have computed the error in approximating the Airy function solution
when $\gamma=h^3$ and $\kappa=15$.  The order of approximation of the Helmholtz equation $q$ is set to 1, 3,  4, and 5 and the corresponding results are respectively marked with diamonds, circles, crosses and squares.

Our goal is to demonstrate that an appropriate choice of $n$ and $q$ can produce high order 
convergence.  Indeed our theory predicts that we should choose $q=n+1$, $n\geq 2$ and expect $O(h^n)$ convergence in the $L^2(\Omega)$ norm since the Airy function solution is smooth and the domain is convex (see Corollary~\ref{uuhcor}).  Results are shown in 
Fig.~\ref{ghqs} and Fig.~\ref{ghcomb}.  

{Fig.~\ref{ghqs} (left panel) demonstrates the need for GPWs in order to obtain high order convergence.  When $q=1$, the GPWs are plane waves and we see no obvious convergence when $n=1$ (three plane waves per element), but at most third order convergence for $n>1$.  This suggests that one approach using an $h$-refinement strategy is to use simple plane waves with $n=3$ to obtain third order convergence under mesh refinement (it appears that $n=3$ offers a useful improvement in accuracy over $n=2$ even if the order of convergence is the same).  To obtain fourth or higher order convergence we need true GPWs with $q>1$. } 

The case $n=1$ is also interesting.  Regardless of $q$ we do not see obvious convergence when $n=1$, whereas for a constant medium the plane wave basis with $n=1$ gives $O(h^2)$ convergence \cite{cessenat_phd}.  The variable refractive index seems to require $n>1$. 
{This is not unreasonable since when $n=1$ the plane waves do not approximate
linear polynomials well, and hence may not converge for a solution corresponding to smoothly varying coefficients.
When $n>1$, piecewise linears are well approximated by plane waves and so we expect (and see) convergence in this case ~\cite{git09,kapita14}.}

{For $n=2$ regardless of the choice $q=1,\cdots,5$ we see $O(h^{3})$ convergence, and for $n=3$ we get $O(h^4)$ convergence provided $q>1$, while if $q=1$ we get $O(h^2)$ convergence. Finally for $n=4$ we only have $O(h^{3})$ convergence when $q=1$, but $O(h^5)$ convergence for $q>1$ (with some deterioration on the finest mesh
when $q=3$ or $q=5$).  This deterioration may be due to the usual conditioning problem experience by plane wave type methods since when $n=4$ the condition number of the system is roughly $10^{20}$ which may impact convergence. The last cases $n=3,4$ confirms that $q$ must increase as $n$ increases although apparently more slowly than we predict.  In addition, with an adequate choice of $q$ we appear to see $O(h^{n+1})$ convergence for $n>1$.  This is the same order as has been found experimentally using $2n+1$ plane waves when $\epsilon$ is constant!~\cite{cessenat_phd}.  An optimal error analysis in that case is also illusive~\cite{buf07}.}

To try to clarify the best relationship between $q$ and $n$ we focus on the cases $q=n-1$, $q=n$ and $q=n+1$ in the left panel of
Fig.~\ref{ghcomb}.  Again the case $n=1$ fails to converge regardless of $q$, but otherwise the most reliable convergence is seen when $n=q$.  


\begin{figure}
\begin{center}
\begin{tabular}{ccc}
\resizebox{.3\textwidth}{!}{\includegraphics{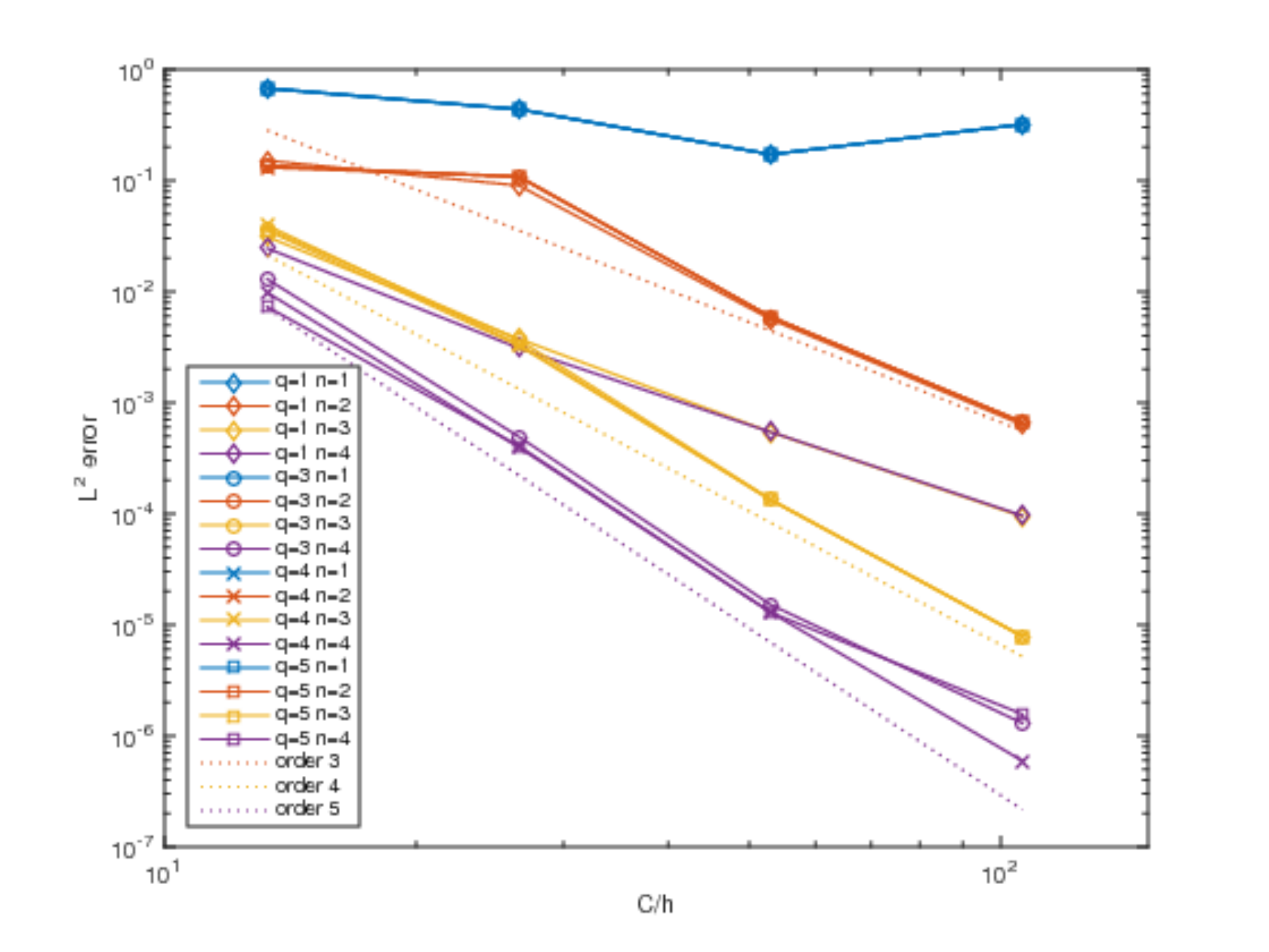}}&\resizebox{.3\textwidth}{!}{\includegraphics{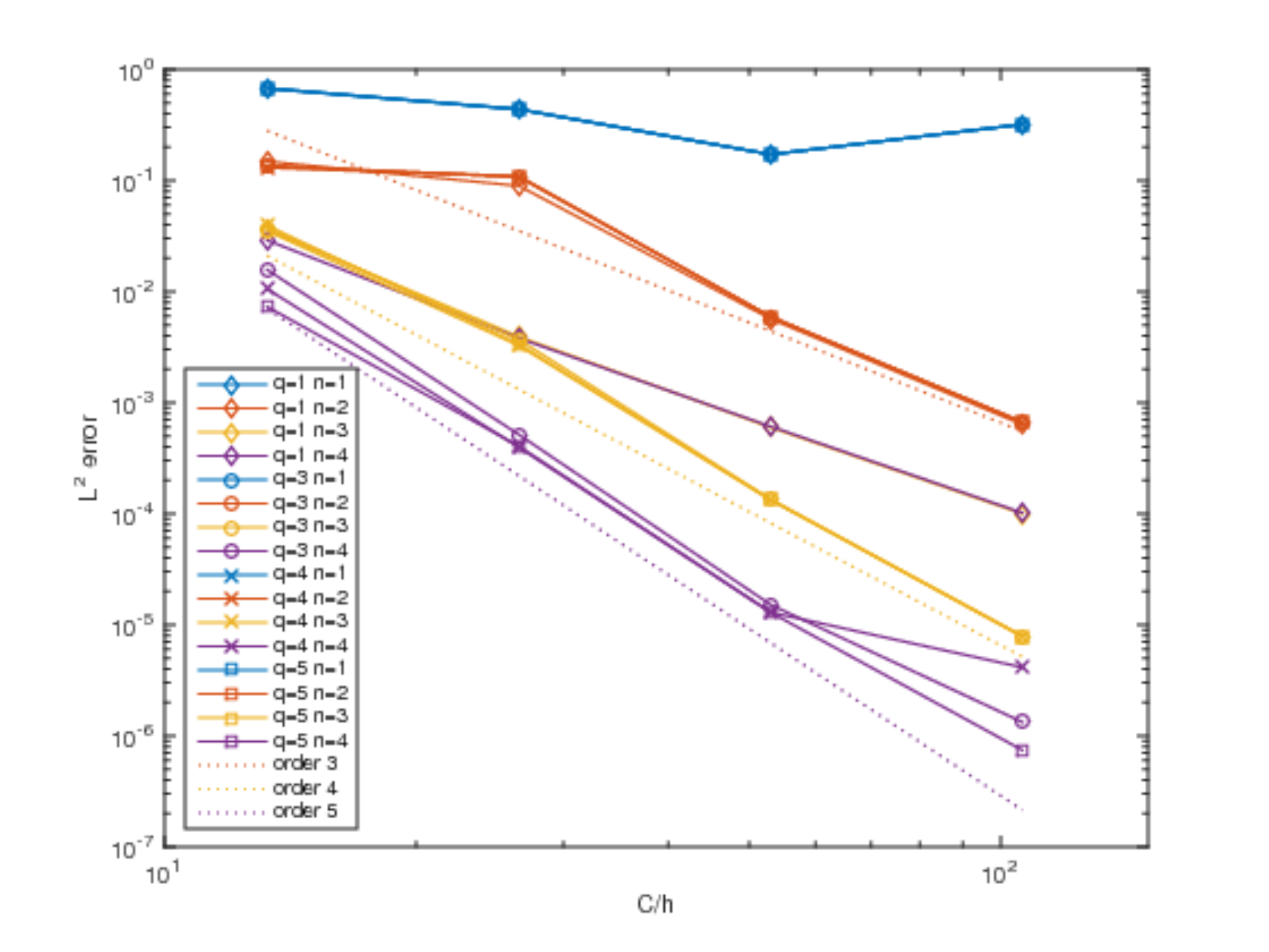}}\resizebox{.3\textwidth}{!}{\includegraphics{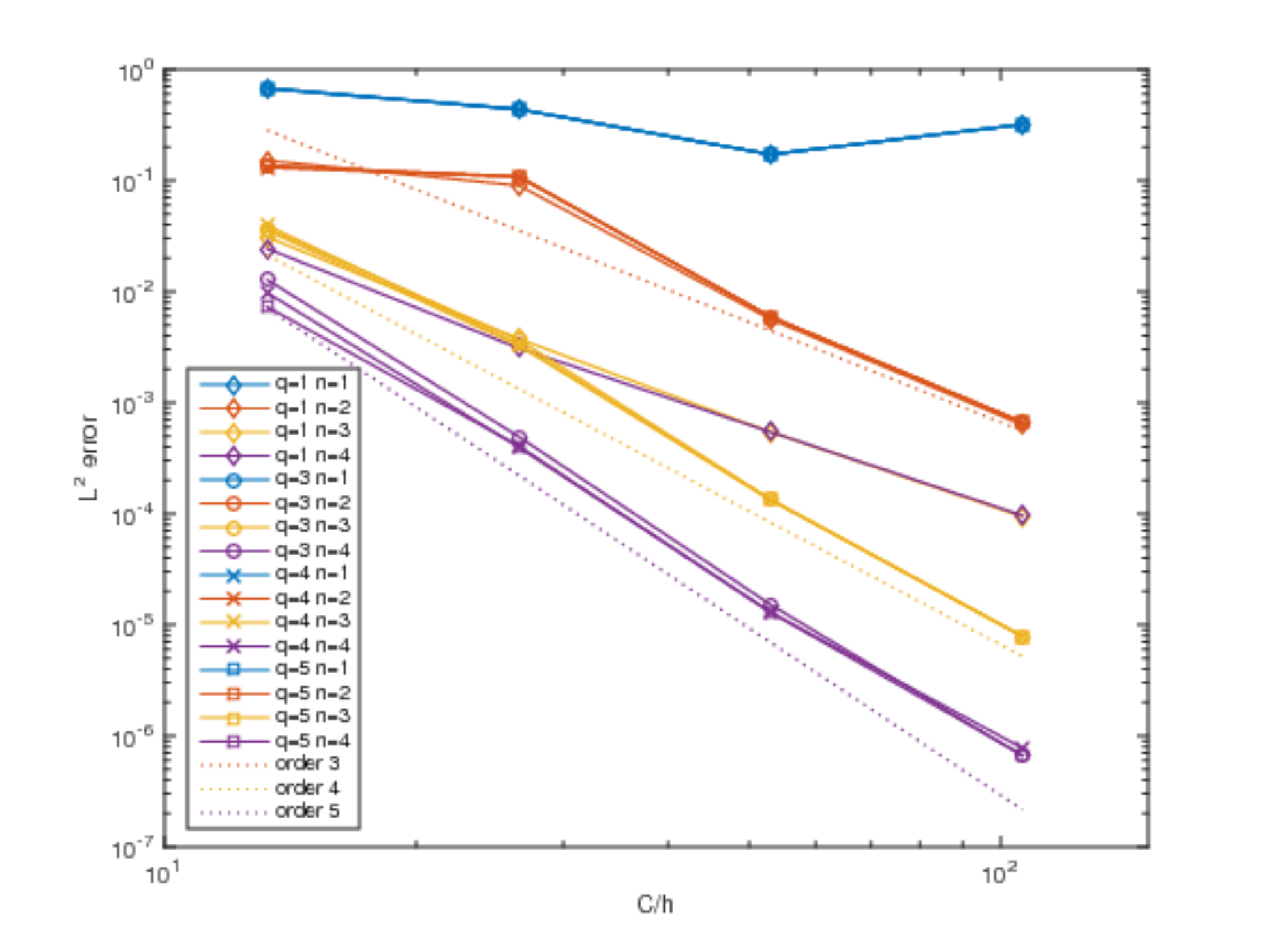}}

\end{tabular}
\end{center}
\caption{$L^2(\Omega)$ norm convergence when $\gamma=h^3$ (left panel), $\gamma=h$ (middle panel) and $\gamma=0$ (right panel).  The dotted lines in each figure show the order of convergence. }
\label{ghqs}
\end{figure}

\begin{figure}
\begin{center}
\begin{tabular}{ccc}
\resizebox{.3\textwidth}{!}{\includegraphics{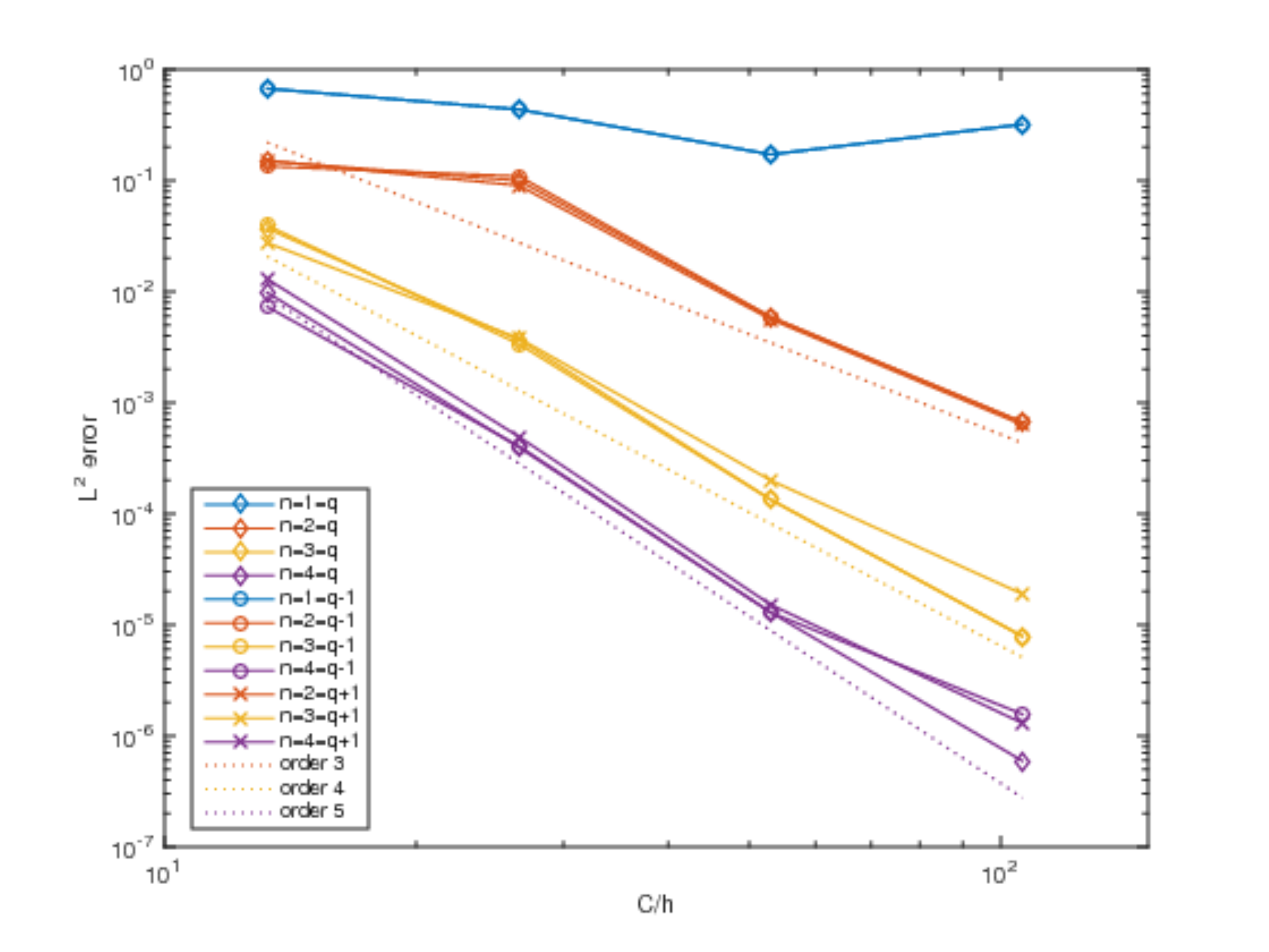}}&\resizebox{.3\textwidth}{!}{\includegraphics{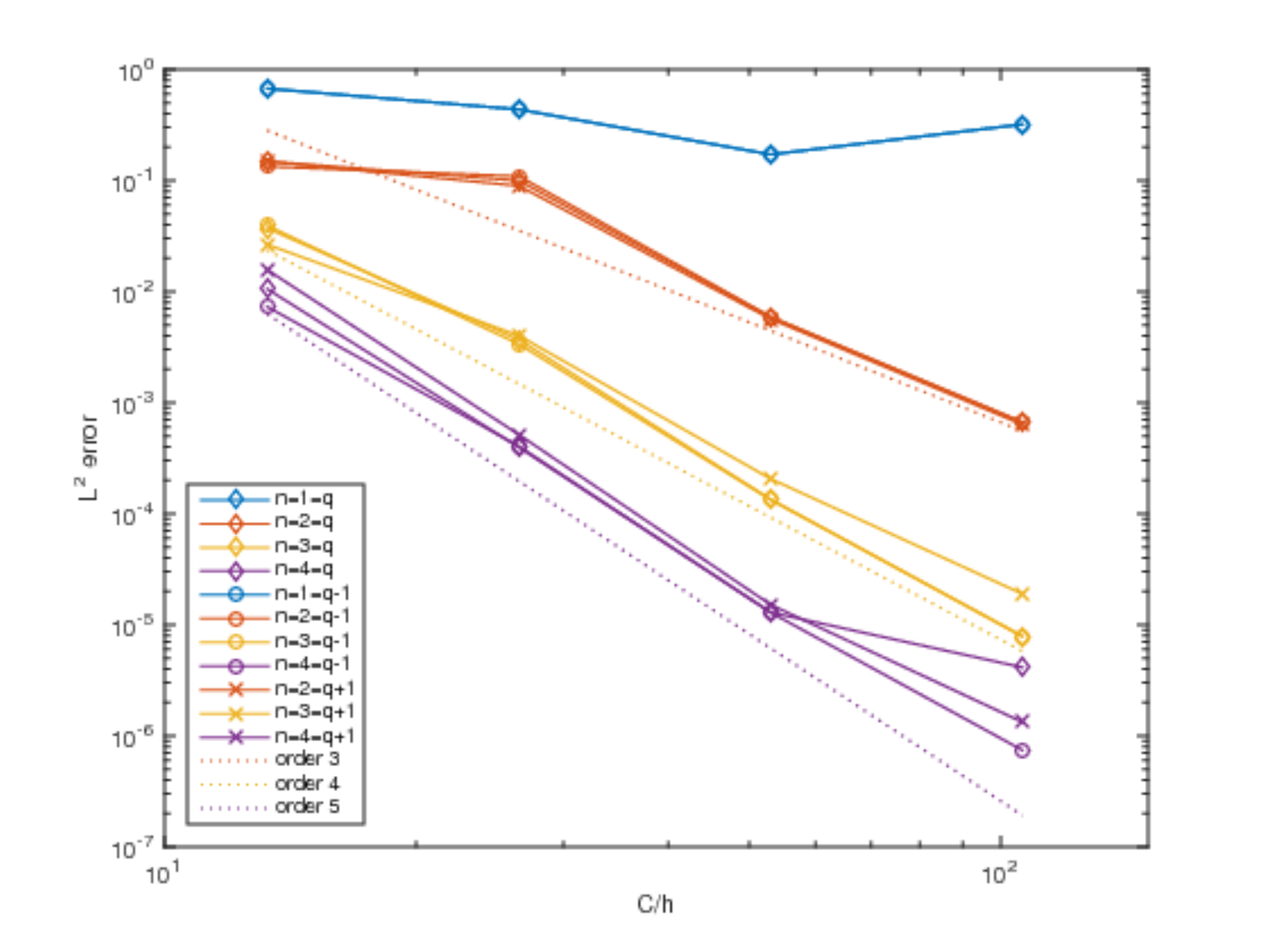}}&
\resizebox{.3\textwidth}{!}{\includegraphics{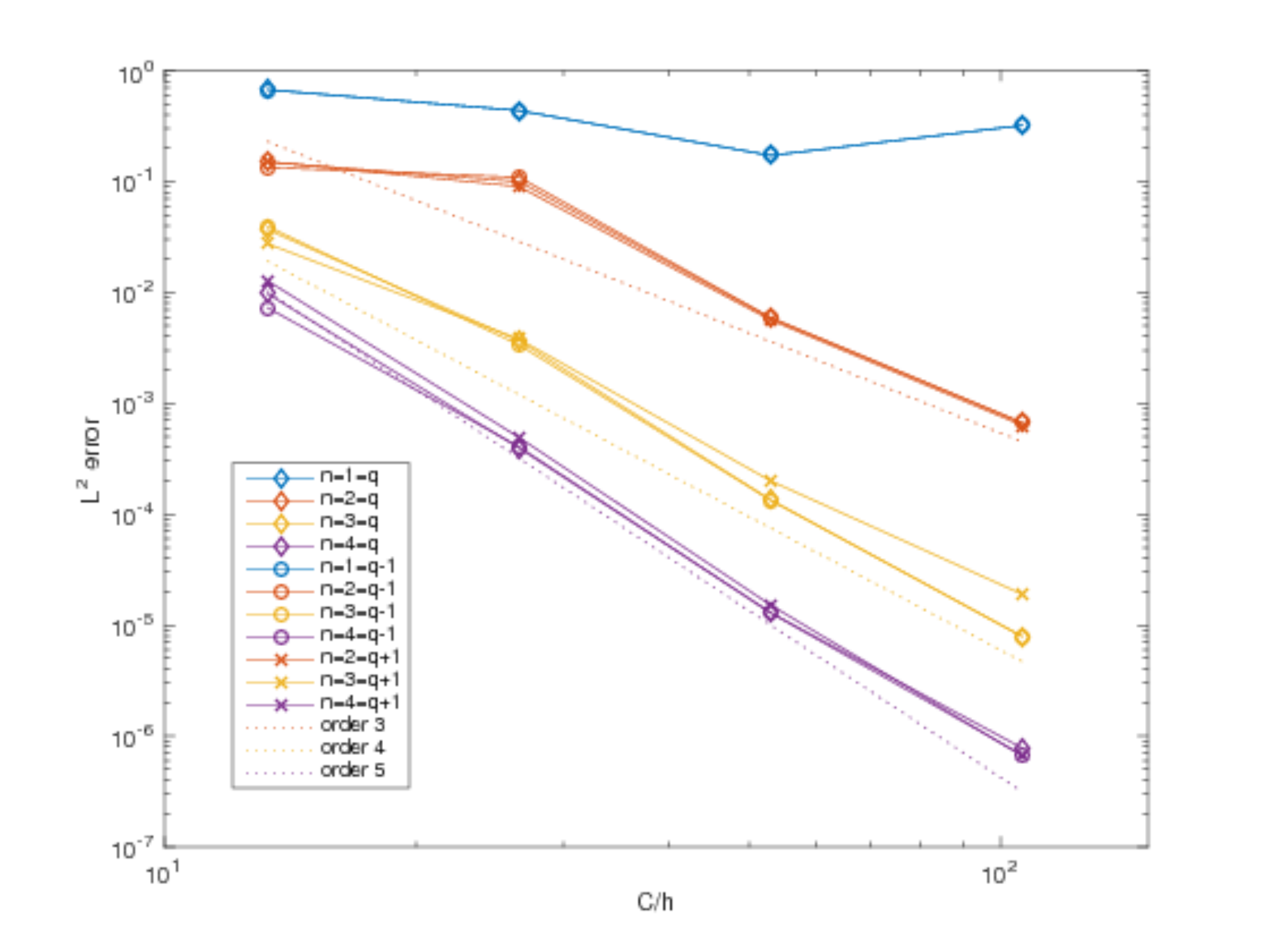}}
\end{tabular}
\end{center}
\caption{$L^2(\Omega)$ norm convergence when $\gamma=h^3$ (left panel), $\gamma=h$ (middle panel) and $\gamma=0$ (right panel). The dotted lines are reference lines showing $O(h^2)$, $O(h^{3})$, $O(h^{4})$, and $O(h^5)$ convergence.  }
\label{ghcomb}
\end{figure}

\subsubsection{The case $\gamma=h$.}
In this section we describe numerical results obtained by increasing the parameter $\gamma$ from $h^3$ to $\gamma =h$, and we compare convergence rates obtained for different combinations of the order of approximation of the equation by the basis functions, $q$, and the number of basis functions per element, $p=2n+1$. This choice for $\gamma$ violates the hypothesis of Theorem \ref{th:L2cv}, but should result in greater stability.



Figure \ref{ghqs} (middle panel) displays series of results for several choices of $q$, while $n$ varies from 1 to 4.  The results are broadly similar to those in the left panel of the same figure, although the case $n=4$ shows a slowing of convergence on fine meshes for $q=3$ and $q=4$.  In this case it appears
that $q=n+1$ is indeed a good choice.

Figure \ref{ghcomb} (middle panel) displays series of results for $n=q$, $n=q-1$, and $n=q+1$,. These convergence studies emphasize the fact that the three choices $n=q$, $n=q-1$, and $n=q+1$ seem to result in approximately the same rate of convergence, suggesting that $q=n-1$ would be the best choice for a fixed value of $n$. The best rates of convergence obtained are 3 for $n=2$, 4 for $n=3$, and 5 for $n=4$.  Although the convergence rates are similar to those when $\gamma=h^3$, the accuracy attained on a given mesh is slightly worse.  This suggests that
choosing $\gamma$ larger than $O(h^3)$ is not useful (other tests, not shown,  with $\gamma=1$ and $\gamma=10^3$ show similar results 
but even worse error at a particular mesh).


\subsubsection{The case $\gamma=0$.}
Our theoretical analysis requires that $\gamma>0$ even to obtain convergence in the DG norm but this term requires integration 
over the interior of all the elements (unlike the standard PWDG or UWVF) and we would prefer to drop it. In addition we saw that
$\gamma=O(h^3)$ gives better results than $\gamma=O(h)$ so we want to test if an even smaller penalty is better. In Fig~\ref{ghqs} and \ref{ghcomb} (right panels)
we show results when $\gamma=0$.

Overall the results are similar to previous results.  Provided $q$ is chosen large enough, we can obtain $O(h^{n+1})$ convergence.  In fact the mesh now seems slightly more stable!

%
}
\subsection{Weber waves}
{
In this section we approximate what we term Weber waves.  These are solutions  of the following problem
\[
\Delta u + \kappa^2\left(\frac{x_2^2}{4}-\frac{a}{\kappa}\right)u=0
\]
in the domain $\Omega=[-1,1]^2$ subject to
\[
u(x_2,y_2)=P_o(\sqrt{\kappa}x_2,a)\mbox{ on }\partial \Omega
\]
where $w(x_2)=P_o(x_2,a)$  is the odd solution of Weber's differential  equation 
\[
\frac{d^2w}{dx_2^2}+\left(\frac{x_2^2}{4}-a\right)=0
\]
defined in \cite{ban04} and implemented in \cite{ban_matlab}.  We choose $a=5$ and $\kappa=50$ which gives the solution in Fig.~\ref{para_exact}.  For this choice of $a$, $\kappa$ and domain, the solution is evanescent for $|x_2|<\sqrt(2/5)$ and oscillatory otherwise.  So this
example again tests how well GPWs can approximate both traveling and evanescent solutions.  

\begin{figure}
\begin{center}
\resizebox{0.5\textwidth}{!}{\includegraphics{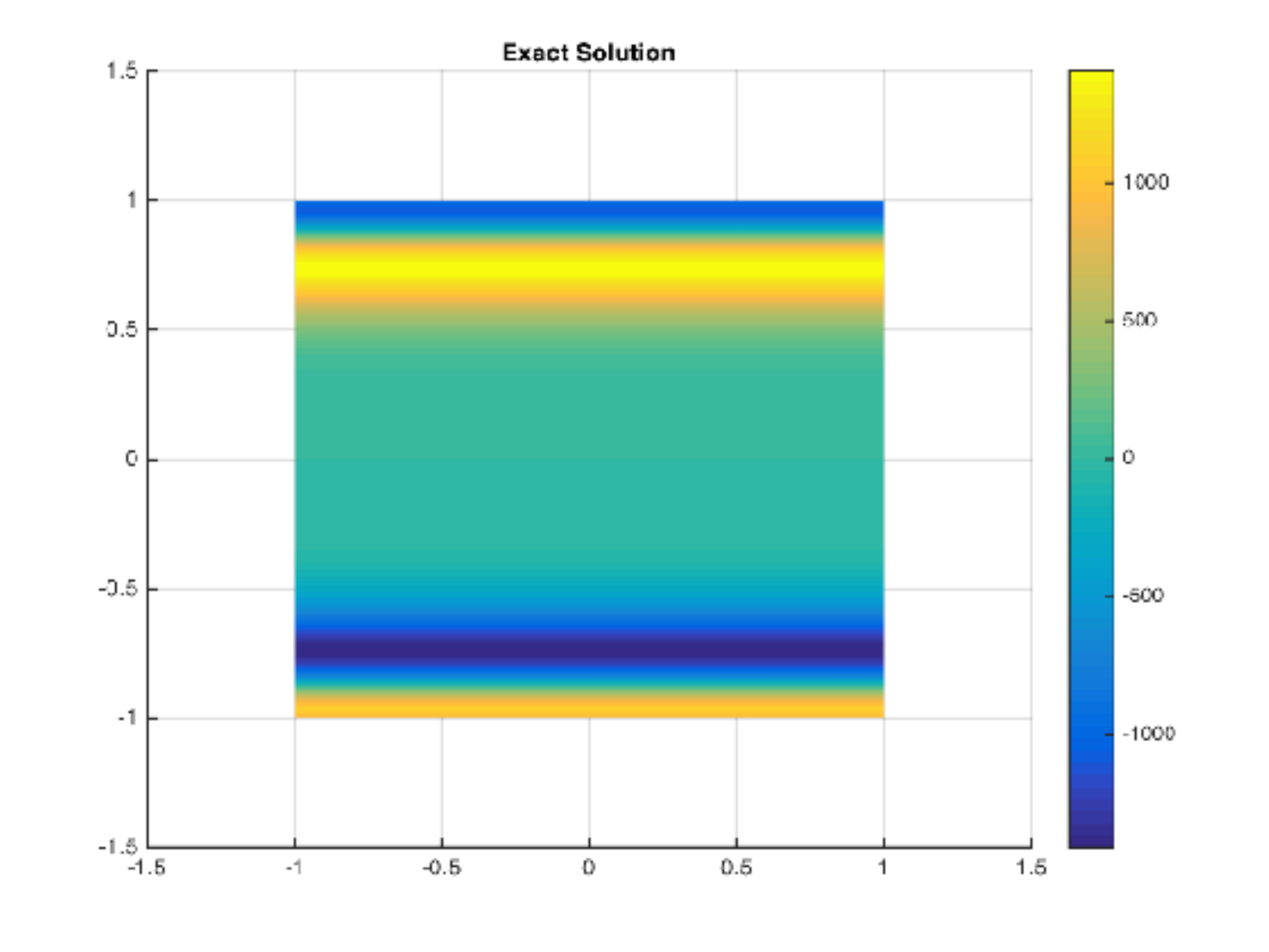}}
\end{center}
\caption{Exact solution for Weber's equation with $a=5$}
\label{para_exact}
\end{figure}

\begin{figure}
\begin{center}
\begin{tabular}{ccc}
\resizebox{.4\textwidth}{!}{\includegraphics{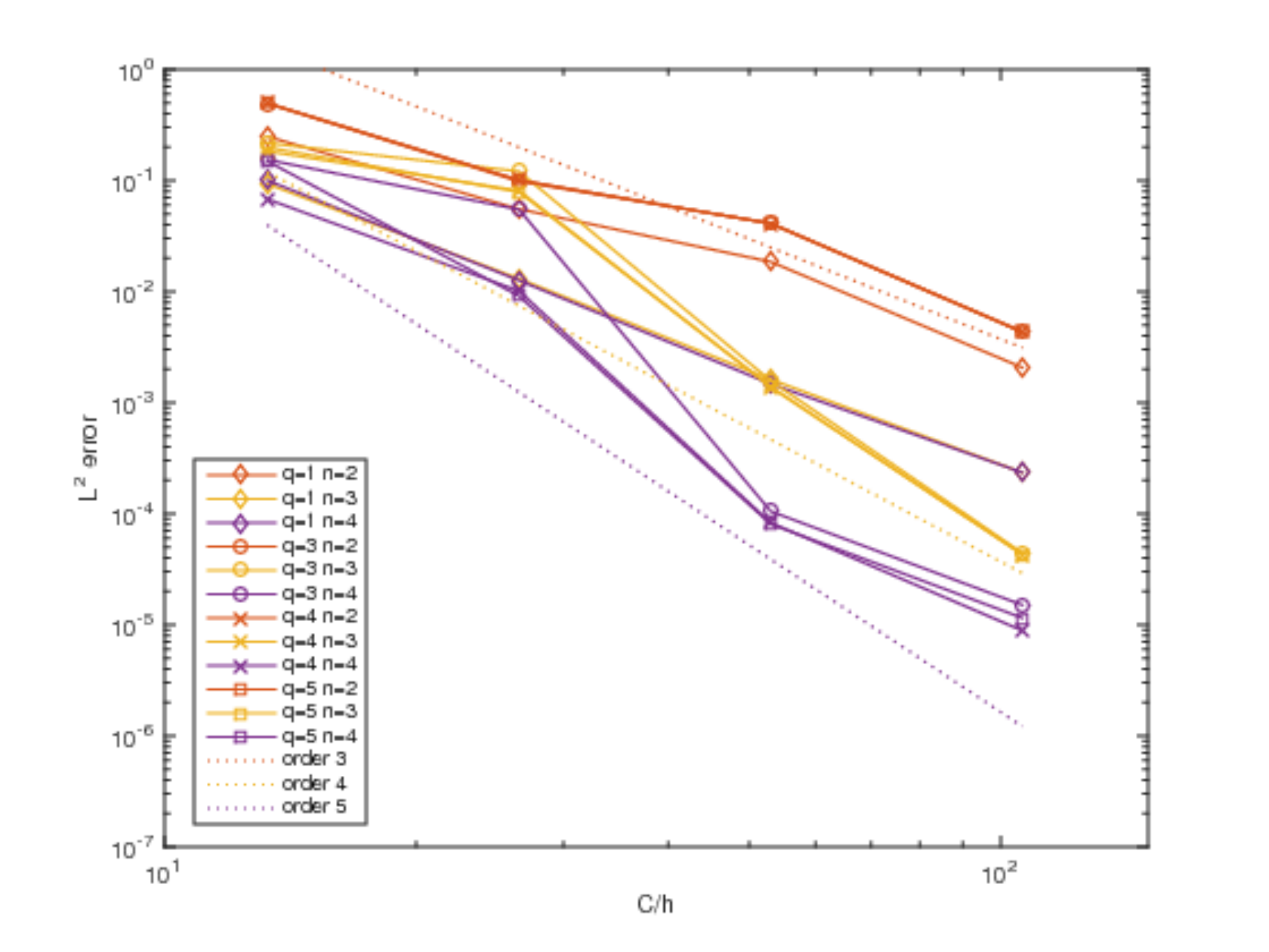}}&\resizebox{.4\textwidth}{!}{\includegraphics{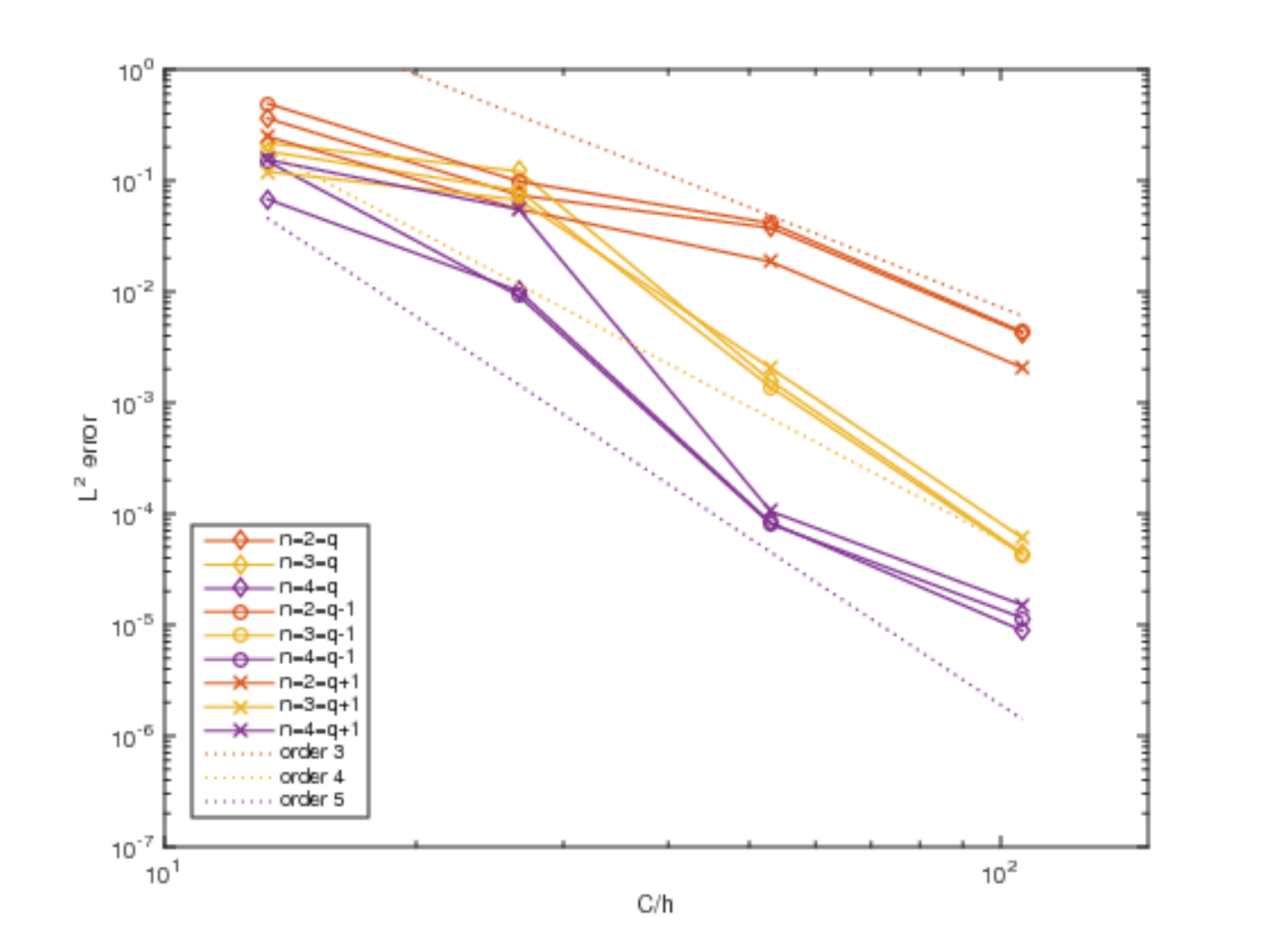}}
\end{tabular}
\end{center}
\caption{Analogues of Fig.~\ref{ghqs} (left) and Fig.~\ref{ghcomb} (right) for the Weber wave example.  }
\label{Wcomb}
\end{figure}

Results are shown in Fig.~\ref{Wcomb}. Broadly the same picture emerges for the Weber example as for the Airy example. We see $O(h^{n+1})$ convergence (this is not completely clear when $n=4$) provided $q$ is large enough.  
\section{Conclusion}
We have provided a modification to the TDG approach that allows the approximation of solutions of the Helmholtz equation in which the refractive index is piecewise smooth using Generalized Plane Waves.  The resulting numerical scheme maintains one advantage of TDG: the number of degrees of freedom per element increases linearly with the order of approximation of the method.  But the method looses one advantage of pure TDG: there is now a need to perform numerical integration element by element.  This is required because we introduce a new stabilization term, and also because the GPW basis functions are not exact solutions of the adjoint problem.  

Theory suggests a choice of parameters that balances polynomial degree with the number of GPWs in the basis element by element.  This is
examined in detail using Airy's equation to provide an exact solution, and substantiated further by using Weber's example.  In the Airy case we have also studied if our new stabilization term is necessary: the numerical results in this one simple case suggest that it can be ignored, but much more testing (for example with less smooth solutions with curved wavefronts) and theoretical backup would be needed to confirm this.  Our testing also suggests that our predicted choice of polynomial degree $q=n+1$ may be excessive.

In summary, we have achieved a first theoretical convergence result for GPWs in a TDG setting.  Our numerical investigations suggest that the theory is not optimal so far, but do show examples where GPWs can provide accurate solutions to wave propagation problems in which the 
coefficients are smooth functions of position.}

\section*{Acknowledgements}
This research reported in this paper was supported in part by NSF grant DMS-1216620.  The authors would like to thank Prof. Timothy Warburton (VPI) for
making is TDG code available to us.

\bibliographystyle{siam}

\bibliography{PwDG1}

\end{document}